\documentclass[12pt, oneside]{article}    	
\usepackage[a4paper,bindingoffset=0.2in,left=1in,right=1in,top=1in,bottom=1in,footskip=.25in]{geometry}   
\usepackage{setspace}
\usepackage{graphicx}	
\usepackage{subcaption}
\usepackage{float}

\usepackage{enumerate}			
\usepackage{enumitem}
\usepackage{amssymb}
\usepackage{amsmath, amsthm, amssymb}

\newtheorem{thm}{Theorem}[section]

\newtheorem{prop}[thm]{Proposition}

\newtheorem{lemma}[thm]{Lemma}
\newtheorem{cor}[thm]{Corollary}

\theoremstyle{definition}
\newtheorem{defi}[thm]{Definition}

\usepackage{tikz}
\newcounter{casenum}

\newcommand*{\rom}[1]{\expandafter{\romannumeral #1\relax}}

\usepackage{authblk}
\makeatletter
\def\@cite#1#2{{\normalfont[{\bfseries#1\if@tempswa , #2\fi}]}}
\makeatother

\title{On the number of linear hypergraphs of large girth}

\author{J\'ozsef Balogh\thanks{Department of Mathematics, University of Illinois at Urbana-Champaign, Urbana, Illinois 61801, USA. Email: jobal@math.uiuc.edu. The first author is partially supported by NSF Grant DMS-1500121, the Arnold O.
Beckman Research Award (UIUC Campus Research Board 15006) and the Langan Scholor Fund (UIUC).
Work was done while the first author was a Visiting Fellow Commoner at Trinity College, Cambridge. } \ \ \ \ \ \ \ Lina Li\thanks{Department of Mathematics, University of Illinois at Urbana-Champaign, Urbana, Illinois 61801, USA. Email: linali2@illinois.edu. }}

\begin{document}
\maketitle

\begin{abstract}
An $r$-uniform \textit{linear cycle} of length $\ell$, denoted by $C_{\ell}^r$, is an $r$-graph with edges $e_1, \ldots, e_{\ell}$ such that for every $i\in [\ell-1]$, $|e_i\cap e_{i+1}|=1$, $|e_{\ell}\cap e_1|=1$ and $e_i\cap e_j=\emptyset$ for all other pairs $\{i, j\},\ i\neq j$. For every $r\geq 3$ and $\ell\geq 4$, we show that there exists a constant $C$ depending on $r$ and $\ell$ such that the number of linear $r$-graphs of girth $\ell$ is at most $2^{Cn^{1+1/\lfloor \ell/2\rfloor}}$. Furthermore, we extend the result for $\ell=4$, proving that there exists a constant $C$ depending on $r$ such that the number of linear $r$-graphs without $C_{4}^r$ is at most $2^{Cn^{3/2}}$. 

The idea of the proof is to reduce the hypergraph enumeration problems to some graph enumeration problems, and then apply a variant of the graph container method, which may be of independent interest. We extend a breakthrough result of Kleitman and Winston on the number of $C_4$-free graphs, proving that
the number of graphs containing at most $n^2/32\log^6 n$ $C_4$'s is at most $2^{11n^{3/2}}$, for sufficiently large $n$. We further show that for every $r\geq 3$ and $\ell\geq 2$, the number of graphs such that each of its edges is contained in only $O(1)$ cycles of length at most $2\ell$, is bounded by $2^{3(\ell+1)n^{1+1/\ell}}$ asymptotically.\\
\end{abstract}

\section{Introduction}
For an integer $r\geq 2$, an $r$-uniform hypergraph (or $r$-graph) $H=(V, E)$ consists of a set $V$ of vertices and a set $E$ of edges, where each edge is an $r$-element subset of $V$. For a family of $r$-graphs $\mathcal{H}$, the \textit{Tur\'an number} (\textit{function}) of $\mathcal{H}$, denoted by $\mathrm{ex}_r(n, \mathcal{H})$, is the maximum number of edges among $r$-graphs on $n$ vertices which contain no $r$-graph from $\mathcal{H}$ as a subgraph. 
Write $\mathrm{Forb}_r(n, \mathcal{H})$ for the set of $r$-graphs with vertex set $[n]$ which contain no $r$-graph from $\mathcal{H}$ as a subgraph. When $\mathcal{H}$ consists of a single graph $H$, we simply write $\mathrm{ex}_r(n, H)$ and $\mathrm{Forb}_r(n, H)$ instead. Since every subgraph of an $H$-free graph is also $H$-free, we have a trivial bound 
\begin{equation}\label{tribound}
2^{\mathrm{ex}_r(n, \mathcal{H})}\leq |\mathrm{Forb}_r(n, \mathcal{H})|\leq \sum_{i\leq \mathrm{ex}_r(n, \mathcal{H})}\binom{\binom {n}{r}}{i}\leq 2n^{r\cdot \mathrm{ex}_r(n, \mathcal{H})}.
\end{equation}

The study on determination of $|\mathrm{Forb}_r(n, H)|$ has a very rich history. 
Recently, the case when $H$ is a linear cycle received more attention.
For integers $r\geq 2$ and $\ell\geq 3$, an $r$-uniform \textit{linear cycle} of length $\ell$, denoted by $C_{\ell}^r$, is an $r$-graph with edges $e_1, \ldots, e_{\ell}$ such that for every $i\in [\ell-1]$, $|e_i\cap e_{i+1}|=1$, $|e_{\ell}\cap e_1|=1$ and $e_i\cap e_j=\emptyset$ for all other pairs $\{i, j\},\ i\neq j$. Kostochka, Mubayi and Verstra\"ete~\cite{KMV1}, and independently, F\"uredi and Jiang~\cite{ZJ} proved that for every $r, \ell\geq 3$, $\mathrm{ex}_r(n, C_{\ell}^r)=\Theta(n^{r-1})$. Then by (\ref{tribound}), we trivially have 
\begin{equation}\label{triboundC}
|\mathrm{Forb}_r(n, C_{\ell}^r)|=2^{\Omega(n^{r-1})} \text{ and } |\mathrm{Forb}_r(n, C_{\ell}^r)|=2^{O(n^{r-1}\log n)}
\end{equation}
for every $r, \ell\geq 3$. Guided and motivated by this development on the extremal numbers of linear cycles, recently, Mubayi and Wang~\cite{MW} showed that $|\mathrm{Forb}_3(n, C_{\ell}^3)|=2^{O(n^2)}$ for all even $\ell$ and improved the trivial upper bound in (\ref{triboundC}) for $r>3$. Inspired by Mubayi and Wang~\cite{MW}'s method, Han and Kohayakawa~\cite{HK} subsequently improved the general upper bound to $2^{O(n^{r-1}\log \log n)}$. Very recently, Balogh, Narayanan and Skokan~\cite{BNS} provided a balanced supersaturation theorem for linear cycles and finally proved $|\mathrm{Forb}_r(n, C_{\ell}^r)|=2^{O(n^{r-1})}$, for every $r, \ell\geq 3$, using the hypergraph container method~\cite{BMS, ST}.

In this paper, we study the enumeration problem of linear hypergraphs containing no linear cycle of fixed length. An $r$-graph $H$ is said to be \textit{linear} if for every $e, e'\in E(H)$, $|e\cap e'|\leq 1$. For a family of linear $r$-graphs $\mathcal{H}$, the \textit{linear Tur\'an number} of $\mathcal{H}$, denoted by $\mathrm{ex}_L(n, \mathcal{H})$, is the maximum number of edges among linear $r$-graphs on $n$ vertices which contain no $r$-graph from $\mathcal{H}$ as a subgraph. Write $\mathrm{Forb}_L(n, \mathcal{H})$ for the set of linear $r$-graphs with vertex set $[n]$ which contain no $r$-graph from $\mathcal{H}$ as a subgraph. Again, when $\mathcal{H}$ consists of a single graph $H$, we simply write $\mathrm{ex}_L(n, H)$ and $\mathrm{Forb}_L(n, H)$ instead. Similarly to (\ref{tribound}), a trivial bound on the size of $\mathrm{Forb}_L(n, H)$ is given as follows.
\begin{equation}\label{triboundL}
2^{\mathrm{ex}_L(n, \mathcal{H})}\leq |\mathrm{Forb}_L(n, \mathcal{H})|\leq \sum_{i\leq \mathrm{ex}_L(n, \mathcal{H})}\binom{\binom {n}{r}}{i}\leq 2n^{r\cdot \mathrm{ex}_L(n, \mathcal{H})}.
\end{equation}
It is known from the famous $(6, 3)$-problem that $n^{2-c\sqrt{\log n}}<\mathrm{ex}_L(n, C_3^3)=o(n^2)$, where the lower bound is given by Behrend~\cite{B} and the upper bound is given by Ruzsa and Szemer\'edi~\cite{RS}.
In 1968, Erd\H os, Frankl and R\" odl~\cite{EFR} showed that for every $r\geq 3$, $\mathrm{ex}_L(n, C_3^r)=o(n^2)$ and $\mathrm{ex}_L(n, C_3^r)=\Omega(n^c)$ for every $c<2$.
Using the so-called 2-fold Sidon sets, Lazebnik and Verstra\"ete~\cite{LV} constructed linear 3-graphs with girth 5 and $\Omega(n^{3/2})$ edges. On the other hand, it is not hard to show that $\mathrm{ex}_L(n, C_4^3)=O(n^{3/2})$. Hence, $\mathrm{ex}_L(n, C_4^3)=\Theta(n^{3/2})$. Kostochka, Mubayi, and Verstra\"ete~\cite{KMVper} proved $\mathrm{ex}_L(n, C_5^3)=\Theta(n^{3/2})$ and conjectured that 
\begin{equation*}
\mathrm{ex}_L(n, C^r_{\ell})=\Theta\left(n^{1+\frac{1}{\lfloor \ell/2 \rfloor}}\right)
\end{equation*}
 for every $r\geq 3$ and $\ell\geq 4.$ Later, Collier-Cartaino, Graber and Jiang~\cite{jiangturan} proved that $\mathrm{ex}_L(n, C^r_{\ell})=O\left(n^{1+\frac{1}{\lfloor \ell/2 \rfloor}}\right)$ for $r\geq 3$ and $\ell\geq 4$. Although the lower bound on the linear Tur\'an number of linear cycles is still far from what is conjectured, following the same logic with the usual Tur\'an problem of cycles, it is natural to guess that 
 \begin{equation}\label{conj}
 |\mathrm{Forb}_L(n, C^r_{\ell})|=2^{\Theta\left(n^{1+\frac{1}{\lfloor \ell/2 \rfloor}}\right)}
 \end{equation}
 for every $r\geq 3$ and $\ell\geq 4$. In this paper, we confirm the above conjecture for $\ell=4$.
\begin{thm}\label{4cycle}
For every $r\geq 3$ there exists $C=C(r)>0$ such that 
\begin{equation*}
|\mathrm{Forb}_L(n, C^r_{4})|\leq2^{Cn^{3/2}}.
\end{equation*}
\end{thm}

The upper bound for $C^3_4$ is sharp in order of magnitude given by $\mathrm{ex}_L(n, C_4^3)=\Theta(n^{3/2})$ and (\ref{triboundL}). In general, since the sharp bound of related linear Tur\'an number remains open, we are not able to confirm the sharpness now.  

For $\ell=3$, the work of Erd\H os, Frankl and R\" odl~\cite{EFR} could be extended to show that $\mathrm{Forb}_L(n, C^r_{3})=2^{o(n^2)}$ for every $r\geq3$.
For $\ell>4$, although we are not ready to prove (\ref{conj}), we provide a result on the girth version. Recall that the girth of a graph is the length of a shortest cycle contained in the graph. 
Kleitman and Wilson~\cite{KWilson}, and independently Kreuter~\cite{kreuter}, and Kohayakawa, Kreuter, and Steger~\cite{KKS} proved that there are $2^{O(n^{1+1/\ell})}$ graphs with no even cycles of length $2\ell$, which made a step towards proving a longstanding conjecture of Erd\H{o}s, who asked for determining the number of $C_{2\ell}$-free graphs. 
Motivated by the above work, we introduce an analogous girth problem on linear hypergraphs.
For a linear $r$-graph $H$, the \textit{girth} of $H$ is the smallest integer $k$ such that $H$ contains a $C^r_{k}$. We remark that for linear $r$-graphs, our girth definition is equivalent to a more classical girth definition, \textit{Berge girth}, i.e. the smallest number $k$ such that the $r$-graph contains a Berge-$C^r_{k}$, as a linear Berge-$C^r_{k}$ must contain a linear cycle of length $i$ for some $3\leq i \leq k$.
For every $r\geq 3$ and $\ell\geq 4$, let $\mathrm{Forb}_L(n, r, \ell)$ denote the set of all linear $r$-graphs on $[n]$ with girth larger than $\ell$. Our second main result is as follows.

\begin{thm}\label{girth}
For every $r\geq 3$ and $\ell\geq 4$, there exists a constant $C=C(r, \ell)>0$ such that $$|\mathrm{Forb}_L(n, r, \ell)|\leq2^{Cn^{1+1/\lfloor\ell/2\rfloor}}.$$
\end{thm}
Recently, Palmer, Tait, Timmons and Wagner~\cite{PTW} considered extremal problems for Berge-hypergraphs and proved our theorem for the case $\ell=4$.
Note that for every $\ell\geq 4$, we have $\mathrm{Forb}_L(n, r, \ell+1)\subseteq \mathrm{Forb}_L(n, r, \ell)$. Therefore, it is sufficient to prove Theorem \ref{girth} for all even $\ell$ and we provide the following equivalent theorem instead.
\begin{thm}\label{evengirth}
For every $r\geq 3$ and $\ell\geq 2$, there exists a constant $C=C(r, \ell)>0$ such that $$|\mathrm{Forb}_L(n, r, 2\ell)|\leq2^{Cn^{1+1/\ell}}.$$
\end{thm}
Once again, the above upper bounds are possibly sharp, but we are not able to confirm it now.

The proofs of Theorems \ref{4cycle} and \ref{evengirth} are based on two graph enumeration results related to even cycles. A classical result of Bondy and Simonovits~\cite{BS} yields $\mathrm{ex}_2(n, C_{2\ell})=O(n^{1+1/\ell})$ for all $\ell\geq 2$. By a series of papers of Kleitman and Winston~\cite{kleitmanwinston}, Kleitman and Wilson~\cite{KWilson}, Kreuter~\cite{kreuter}, Kohayakawa, Kreuter, and Steger~\cite{KKS}, and Morris and Saxton~\cite{MS}, we now know that the number of $C_{2\ell}$-free graphs is at most $2^{O(n^{1+1/\ell})}$. 
Inspired by these works, we prove that the number of graphs containing some but not many short cycles is still at most $2^{O(n^{1+1/\ell})}$, which may be of independent interest. We state our results as follows.
\begin{thm}\label{maintheoremC4}
Let $n$ be a sufficiently large integer and $a=32\log^6 n$. The number of $n$-vertex graphs with at most $n^2/a$ 4-cycles is at most $2^{11n^{3/2}}.$
\end{thm}

Given a graph $G$ on $[n]$, for every integer $k\geq 3$ and every edge $uv\in E(G)$, denote by $c_{k}(u, v; G)$, the number of $k$-cycles in $G$ containing edge $uv$. When the underlying graph is clear, we simply write $c_{k}(u, v)$. For an integer $\ell\geq 3$ and a constant $L>0$, write $\mathcal{G}_n(\ell, L)$ for the family of graphs $G$ on $[n]$ such that for every $3\leq k\leq \ell$ and $uv\in E(G)$, $c_{k}(u,v; G)\leq L$. 

\begin{thm}\label{mainthmlonger}
For an integer $\ell\geq 3$ and a constant $L>0$, let $n$ be a sufficiently large integer and then we have $$|\mathcal{G}_n(2\ell, L)|\leq 2^{3(\ell+1)n^{1+1/\ell}}.$$
\end{thm}

Like many of these advances, our approach to proving Theorems~\ref{maintheoremC4} and~\ref{mainthmlonger} relies on the graph container method developed in~\cite{kleitmanwinston}, in which one assigns a \textit{certificate} for each target graph. The certificate should be able to uniquely determine the target graph, and then we can estimate the number of certificates instead of graphs. However, the previous applications of the graph container method address the problems for graphs forbidding short cycles, while we concern with the graphs with sparse short cycles.
Therefore, the means by which we apply this technique is quite non-standard, and requires some new ideas.

It is not hard to extend Theorem \ref{maintheoremC4} to $a=\Theta(\log^5 n)$ by proving a similar statement for $\mathcal{G}_n(4, \sqrt{n}/\log^4 n)$ as in Theorem~\ref{mainthmlonger}. We choose to display the proof of Theorem~\ref{maintheoremC4} since it contains some ideas which may bring more insights of this method to readers.
Let $p=\omega/(\sqrt{n}\log n)$. Note that the number of graphs on $[n]$ with $p\binom{n}{2}$ edges is about $2^{\omega n^{3/2}}$ and they typically contain $\Theta(n^4p^4)=\Theta(\omega^4 n^2/\log^4 n)$ 4-cycles. Therefore, $a=\Theta(\log^4 n)$ would be the best possible in Theorem \ref{maintheoremC4} and we believe that it should be the truth.
Given by the connection between Sidon sets and graphs without 4-cycles, this problem is closely related with an enumeration problem on generalized Sidon set which was recently studied in the authors' another paper~\cite{BL}.
\newline

Throughout this paper, we let $[n]$ denote the set $\{1, 2, \ldots, n\}$. For a graph $G$ and a set $S\subseteq V(G)$, the \textit{induced subgraph} $G[S]$ is the subgraph of $G$ whose vertex set is $S$ and whose edge set consists of all of the edges with both endpoints in $S$. Let $\delta(G)$ denote the minimum degree of graph $G$ and $\Delta(G)$ denote the maximum degree of $G$. 
For a multigraph $G$ and a vertex $v\in V(G)$, the \textit{neighborhood} $N_G(v)$ of $v$ is the set of all vertices adjacent to $v$ in $G$ and the \textit{degree} $d_G(v)$ of $v$ is the number of edges incident to $v$ in $G$. For a set $S\subseteq V(G)$, the \textit{neighborhood of $v$ restricted to $S$} is $N_G(v, S)=N_G(v)\cap S$; the \textit{degree of $v$ restricted to $S$}, denoted by $d_G(v, S)$, is the number of edges incident to $v$ with another endpoint in $S$. When the underlying graph is clear, we simply write $N(v, S)$ and $d(v, S)$ instead. All logarithms have base 2.

\section{Graphs containing a few 4-cycles}
\subsection{Preliminary results}
\begin{defi}[Min-degree ordering, Min-degree sequence]
For  a graph $G$ on $[n]$, a \textit{min-degree ordering} is an ordering $v_n < v_{n-1} <\ldots < v_1$, such that $v_i$ is a vertex of minimum degree in the graph $G_i= G[v_i,\ldots,v_1]$, for every $i\in [n]$ (if there are more than one vertices of the minimum degree, choose the one with the largest label). Let $d_i=d_{G_i}(v_i)$, then $d_n, d_{n-1},\ldots, d_1$ is called the \textit{min-degree sequence}.
\end{defi}
\begin{lemma}\label{extremal}
Let $G$ be an $n$-vertex graph with average degree $d$. If $d\geq 2\sqrt{n}$, then $G$ contains at least $d^4/36$ copies of 4-cycles.
\end{lemma}
\begin{proof}
Let $v_1, v_2,\ldots, v_n$ be the vertices in $G$ and $b_i=d_{G}(v_i)$ for every $i\in [n]$. Let $S$ be the set of paths of length 2 (or 3-paths) in $G$. We will count 3-paths in two ways. 

First, for a vertex $v_i$, the number of 3-paths containing $v_i$ as the middle point is exactly $\binom{b_i}{2}$. Therefore, we have
$$|S|=\sum_{i=1}^n\binom{b_i}{2}\geq n\binom{(\sum_{i=1}^nb_i)/n}{2}=n\binom{d}{2}\geq \frac13d^2n.$$
On the other hand, for $1\leq i< j\leq n$, let $c_{ij}$ be the number of common neighbors of $v_i$ and $v_j$. Then $|S|=\sum_{1\leq i<j\leq n}c_{ij}$. Therefore, the number of 4-cycles in $G$ is equal  to
$$\frac12\sum_{1\leq i<j\leq n}\binom{c_{ij}}{2}\geq\frac12\binom n2\binom{(\sum_{i<j}c_{ij})/\binom n2}{2}=\frac12\binom n2\binom{|S|/\binom n2}{2}\geq \frac{|S|^2}{4n^2}\geq \frac{d^4}{36}.$$ 
\end{proof}
From Lemma~\ref{extremal}, we immediately obtain the following corollary. 

\begin{cor}\label{degreecondition}
Let $G$ be a $n$-vertex graph which contains at most $4n^2/9$ 4-cycles, and $d_n, \ldots, d_1$ be the min-degree sequence of $G$. Then for every $i\in [n]$, $$d_i\leq 2\sqrt{n}.$$
\end{cor}
\begin{proof}
Suppose that there exists $k\in [n]$, such that $d_k>2\sqrt{n}.$ 
Then by Lemma~\ref{extremal}, the number of 4-cycles in $G_k$ is at least $d_k^4/36>\frac{4}{9}n^2,$ which contradicts our assumption.
\end{proof}

We also provide an estimation for the following binomial coefficients, which will be used repeatly  later.
\begin{lemma}\label{optimize}
For integers $n, k, \ell$ and a constant $c$ satisfying  $c n/k^{\ell}\geq k$,
\[
\binom{c n/k^{\ell}}{k}\leq 2^{\frac{\ell + 1}{2^{1/\ln 2}\ln 2}(cen)^{\frac{1}{\ell + 1}}},
\]
where $2^{1/\ln 2}\ln 2\approx 1.88.$
\end{lemma}
\begin{proof}
Let $f(x)=\left(\log cen - (\ell + 1)\log x\right)x$ on $(0, +\infty)$. Since $f(x)$ is a concave function, it is maximized at the point $x^*$, where $f'(x^*)=\log cen - \frac{\ell + 1}{\ln 2} - (\ell + 1)\log x^*=0$, i.e. $\log x^* =\frac{ \log cen }{\ell + 1} - \frac{1}{\ln 2}.$ Therefore, we have
\[
f(k)\leq f(x^*)= \left(\log cen - (\ell + 1)\left(\frac{ \log cen }{\ell + 1} - \frac{1}{\ln 2}\right)\right)2^{\frac{ \log cen }{\ell + 1} - \frac{1}{\ln 2}}
=\frac{\ell + 1}{2^{1/\ln 2}\ln 2}(cen)^{\frac{1}{\ell + 1}}.
\]
Since $\binom{n}{k}\leq \left(\frac{ne}{k}\right)^{k}$ for every $1\leq k \leq n$, we obtain that 
\[
\binom{c n/k^{\ell}}{k}\leq \left(\frac{cen}{k^{\ell + 1}}\right)^{k}
=2^{f(k)}
\leq 2^{\frac{\ell + 1}{2^{1/\ln 2}\ln 2}(cen)^{\frac{1}{\ell + 1}}}.
\]
\end{proof}
\subsection{Certificate lemma}
This section is devoted to prove our main lemma, which is a key step to build the certificates for graphs with sparse 4-cycles.
This lemma can be viewed as a generalization of the Kleitman-Winston algorithm~\cite{kleitmanwinston}, which builds certificates for graphs without 4-cycles.
Before we proceed, we first need a counting lemma, which will be used later in the proof. 

For a graph $F$, denote by $F^2$ the multigraph defined on $V(F)$ such that for every distinct $u, v\in V(F^2)$, the multiplicity of $uv$ in $F^2$ is the number of $(u, v)$-paths of length $2$ in $F$.

\begin{lemma}\label{countinglemma}
For integers $n>m\geq d\geq 8$, let $F$ be an $m$-vertex graph with $\delta(F)\geq d-1$ and $H=F^2$. Then for every $J\subseteq V(H)$ of size at least $4n/d$, we have $$e(H[J])\geq \frac{d^2|J|^2}{4n}.$$ 
\end{lemma}
\begin{proof}
Write $V(F)=\{v_1, \ldots, v_m\}$. For every $j\in[m]$, let $b_j=d_{F}(v_j, J)$. Then we have $\sum_{j=1}^mb_j=\sum_{v\in J}d_{F}(v)\geq |J|(d-1)\geq \frac{4(d-1)}{d}n>3n>3m$. Therefore, we obtain that 
$$e(H[J])=\sum_{j=1}^{m}\binom{b_j}{2}\geq m\binom{\frac{\sum b_j}{m}}{2}\geq m\binom{\frac{|J|(d-1)}{m}}{2}\geq  \frac{|J|^2(d-1)^2}{3m}\geq\frac{d^2|J|^2}{4n}.$$
\end{proof}

\begin{lemma}[Certificate lemma]\label{mainlemmaC4}
For a sufficiently large integer $n$, define $b=16\log^4 n$ and $g=32\log^5 n$. Let $m$ and $d$ be the integers satisfying $m\leq n-1$ and $\frac{\sqrt{n}}{\log n}\leq d\leq 2\sqrt{n}$. 
Suppose that $F$ is an $m$-vertex graph with $\delta(F)\geq d-1$ and $H=F^2$. 
Additionally, assume that for every $u, v\in V(F)$, $|N_F(u)\cap N_F(v)|\leq \sqrt{n}/b$. Then for every set $I\subseteq V(F)$ of size $d$ which satisfies $e(H[I])\leq n/g$, there exist a set $T$ and a set $C(T)$ depending only on $T$, not on $I$, such that
\begin{enumerate}[label={\upshape(\roman*)}]
  \item $T\subseteq I\subseteq C(T)$,
  \item $|T|\leq2\sqrt{n}/\log n$,
  \item $|C(T)|\leq 5n/d$.
\end{enumerate}
\end{lemma}
\begin{proof}
Let $I$ be a subset of $V(F)$ of size $d$ which satisfies $e(H[I])\leq n/g$. Following the ideas of Kleitman and Winston~\cite{kleitmanwinston}, we describe a deterministic algorithm that associates to the set $I$ a pair of sets $T$ and $C(T)$, which shall be treated as the `fingerprint' and the `container' respectively. 

Let $I_h=\{v\in I:\ d_H(v, I)> \sqrt{n}/b\}$ and $I_l=\{v\in I:\ d_H(v, I)\leq \sqrt{n}/b\}$. Since $e(H[I])\leq n/g$, the size of $I_h$ is at most $$\frac{2e(H[I])}{\sqrt{n}/b}\leq\frac{2\sqrt{n}\cdot b}{g}=\frac{\sqrt{n}}{\log n},$$ which is sufficiently small.
Therefore, we only need to concern the vertices in $I_{l}$.\\

\noindent\textbf{The core algorithm.} We start the algorithm with sets $A_0=V(H)-I_h$, $T_0=\emptyset$ and the function $t_0(v)=0$, for every $v\in V(H)-I_h$. 
As the algorithm proceeds, one should view $A_i$ as the set of `candidate' vertices, $T_i$ as the set of `representive' vertices, and $t_i(v)$ as a `state' function which is used to control the process.
In the $i$-th iteration step, we pick a vertex $u_i\in A_i$ of maximum degree in $H[A_i]$. In case there are multiple choices, we give preference to vertices that come earlier in some arbitrary predefined ordering of $V(H)$ as we always do, even if it is not pointed out at each time. If $u_i\in I_l$, we define
  $$t_{i+1}(v)=\left\{\begin{array}{ll} 
    t_i(v)+d_H(v, u_i) & \text{if } v\in A_{i},\\
    t_i(v)  & \text{if } v\notin A_{i},
    \end{array}\right.$$
and $Q=\{v \mid t_{i+1}(v)>\sqrt{n}/b\}$, and let $T_{i+1}=T_i+u_i$, $A_{i+1}=A_i-u_i-Q$. Otherwise, let $T_{i+1}=T_i$, $A_{i+1}=A_i - u_i$ and $t_{i+1}(v)=t_i(v)$, for every $v\in V(H)-I_h$. 
The algorithm terminates at step $K$ once we get a set $A_{K}$ of size at most $4n/d$.
We also assume that $u_{K-1}\in T_K$ as otherwise we can continue the algorithm until it is satisfied.\\

The algorithm outputs a vertex sequence $\{u_1, u_2, \ldots, u_{K-1}\}$, a set of `representive' vertices $T_{K}$ and a strictly decreasing set sequence $\{A_0, A_1, A_2, A_3, \ldots, A_{K}\}$. Let 
\[
T=T_{K}\cup I_h, \quad \text{and}\quad C(T)=A_{K}\cup T.
\]
From the algorithm, we have $T_{K}\subseteq I_l$ and therefore $T\subseteq I$. 
Furthermore, if a vertex $v$ satisfies $t_{i}(v)>\sqrt{n}/b$ for some $i$, then we have $d_H(v, I)\geq t_i(v)>\sqrt{n}/b$, which implies $v\notin I_l$. Therefore, we maintain  $I_l\subseteq A_i\cup T_i$ for every $i\leq K$ and in particular we have $I\subseteq A_{K}\cup T_{K}\cup I_h =A_{K}\cup T=C(T)$. Hence, Condition (i) is satisfied.
Another crucial fact is that $C(T)$ depends only on $T$, not on $I$. The reason is that for a given underlying graph, its max degree sequence is fixed once we break the tie by some predefined ordering on vertices. Therefore, for two sets $I_1, I_2$ with the same `fingerprint' $T$, the algorithm outputs the same vertex sequence $\{u_1, u_2, \ldots, u_{K-1}\}$, which uniquely determines the set $C(T)$ by the mechanics of the algorithm.

To verify Conditions (ii) and (iii), it is sufficient to show that $|T_{K}|\leq \sqrt{n}/\log n$. Once we prove it, we immediately obtain $$|T|=|T_{K}|+|I_h|\leq \frac{\sqrt{n}}{\log n}+\frac{\sqrt{n}}{\log n}=\frac{2\sqrt{n}}{\log n},$$
and $$|C(T)|=|A_{K}|+|T|\leq \frac {4n}{d}+\frac{2\sqrt{n}}{\log n}\leq \frac {5n}{d},$$ completing the proof.

Denote $q$ the integer such that $n/2^{q}\leq|A_{K}|< n/2^{q-1}$. By the choice of $A_K$, we have $q<\log n$. For every integer $1\leq l\leq q$, define $A^l$ to be the first $A$-set satisfying 
$$\frac{n}{2^{l}}\leq|A^l|< \frac{n}{2^{l-1}},$$
 if it exists, and let $T^l$ be the corresponding $T$-set and $t^l(v)$ be the corresponding $t$-function of $A^l$ . Note that $A^l$ may not exist for every $l$, but $A^{q}$ always exists and it could be that $A^q=A_{K}.$ Suppose that
 $$A^{l_1}\supset A^{l_2}\supset\ldots\supset A^{l_{p}}$$ 
 are all the well-defined $A^l$, where $p\leq q$. By the above definition, we have $A^{l_1}=A_0$, $T^{l_1}=T_0$ and $l_p=q$. Define $A^{l_{p+1}}=A_{K}, T^{l_{p+1}}=T_{K}$. Now, we have 
 \begin{equation}\label{TK}
 T_{K}=\bigcup_{j=2}^{p+1}(T^{l_j}-T^{l_{j-1}}).
 \end{equation}
 To achieve our goal, we are going to estimate the size of $T^{l_j}-T^{l_{j-1}}$ for every $2\leq j\leq p+1$.

From the algorithm, we have $t^{l_j}(v)\leq \sqrt{n}/b$, for every $v\in A^{l_j}\cup T^{l_j}$. Moreover, for $v\in A^{l_{j-1}}-A^{l_j}-T^{l_j}$, suppose that $v$ is removed in step $i$, then we have  
$$t^{l_j}(v)\leq t_{i-1}(v)+d_H(v, u_i)\leq \frac{\sqrt{n}}{b}+|N_F(u_i)\cap N_F(v)|\leq \frac{2\sqrt{n}}{b},$$ 
where $u_i$ is the selected vertex in step $i$.
Therefore, we obtain
\begin{equation}\label{ub}
\sum_{v\in A^{l_{j-1}}}t^{l_j}(v)\leq\frac{2\sqrt{n}}{b}|A^{l_{j-1}}|\leq \frac{2n^{3/2}}{2^{l_{j-1}-1}b}. 
\end{equation}

Let $2\leq j\leq p$. For every $u_i\in T^{l_j}-T^{l_{j-1}}$, $u_i$ is chosen of maximum degree in $H[A_i]$, where $A_i$ is a set between $A^{l_{j-1}}$ and $A^{l_j}$.  By the choice of $A^{l_j}$, we have $|A_i|\geq n/2^{l_{j-1}}$. By Lemma $\ref{countinglemma}$, we have 
$$d_H(u_i, A_i)\geq \frac{d^2|A_{i}|}{4n}\geq \frac{d^2}{2^{l_{j-1} + 2}}.$$
Note that $d_H(u_i, A_i)$ only contributes to $t^{l_j}(v)$ for $v\in A_i\subseteq A^{l_{j-1}}$. Then we obtain 
 \begin{equation}\label{lb}
 \left|T^{l_j}-T^{l_{j-1}}\right|\frac{d^2}{2^{l_{j-1} + 2}}\leq\sum_{u_i\in T^{l_j}-T^{l_{j-1}}}d_H(u_i, A_i)\leq\sum_{v\in A^{l_{j-1}}}t^{l_j}(v).
\end{equation}
Combining (\ref{ub}) and (\ref{lb}), we have
\begin{equation*}
 \left|T^{l_j}-T^{l_{j-1}}\right|\frac{d^2}{2^{l_{j-1} + 2}}\leq\frac{2n^{3/2}}{2^{l_{j-1}-1}b},
\end{equation*}
which implies $$ \left|T^{l_j}-T^{l_{j-1}}\right|\leq \frac{16n^{3/2}}{bd^2}\leq\frac{16\sqrt{n}\log^2 n}{b}=\frac{\sqrt{n}}{\log^2 n}$$
for $2\leq j\leq p$.
For $j=p+1$, since we have $\frac{n}{2^q}\leq |A^{l_{p+1}}|\leq|A^{l_p}|\leq\frac{n}{2^{q-1}}$, by a similar argument, we obtain that
$$\left|T^{l_{p+1}}-T^{l_{p}}\right|\frac{d^2|A^{l_{p+1}}|}{4n}\leq \sum_{u_i\in T^{l_{p+1}}-T^{l_{p}}}d(u_i, A_i)\leq \sum_{v\in A^{l_{p}}}t^{l_{p+1}}(v)\leq \frac{2\sqrt{n}}{b}|A^{l_{p}}|,$$
which gives 
$$\left|T^{l_{p+1}}-T^{l_{p}}\right|\leq\frac{16n^{3/2}}{bd^2}\leq\frac{16\sqrt{n}\log^2 n}{b}=\frac{\sqrt{n}}{\log^2 n}.$$
Finally, by (\ref{TK}), we get $$|T_{K}|=\bigcup_{j=2}^{p+1}|T^{l_j}-T^{l_{j-1}}|\leq p\cdot\frac{\sqrt{n}}{\log^2 n}\leq q\cdot\frac{\sqrt{n}}{\log^2 n}\leq\frac{\sqrt{n}}{\log n}.$$
\end{proof}
\subsection{Proof of Theorem \ref{maintheoremC4} }
In this section, we give an upper bound on the number of graphs containing only `few' 4-cycles. Before we proceed to prove Theorem \ref{maintheoremC4}, we need to do a cleaning process for the target graphs in order to apply Lemma \ref{mainlemmaC4}.

Let $a=32\log^6 n$, $g=32\log^5 n$ and $b=16\log^4 n.$ Given a graph $G$ on $[n]$, for every $1\leq i<j\leq n$, define $N_G(i, j)$ to be the set of common neighbors of $i$ and $j$ in $G$. Let
\begin{equation*}
m_G(i, j)=\left\{
\begin{array}{lll}
|N_G(i, j)| &&  \textrm{when}\ |N_G(i, j)|> \frac{\sqrt{n}}{b},\\
0 && \textrm{when}\ |N_G(i, j)|\leq \frac{\sqrt{n}}{b}.
\end{array}\right.
\end{equation*}
We delete all edges from $i$ to $N_G(i,j)$, for all $1\leq i<j\leq n$ with $m_G(i, j)\neq 0$. Then the resulting subgraph, denoted by $\widehat{G}$, satisfies  $|N_{\widehat{G}}(i, j)|\leq \sqrt{n}/b$, for every $1\leq i<j\leq n$. 
Let $\mathcal{G}_n$ be the family of graphs on $[n]$ with at most $n^2/a$ 4-cycles and $\mathcal{\widehat{G}}_n=\{\widehat{G}: G\in \mathcal{G}_n\}.$
\begin{lemma}\label{cleangraph}
Let $n$ be a sufficiently large integer. Then for every $G\in\mathcal{G}_n$, we have $$|E(G)-E(\widehat{G})|\leq \frac{4n^{3/2}}{\log^2 n}.$$
\end{lemma}
\begin{proof}
By counting 4-cycles in $G$, we obtain that
$$\frac12\sum_{i<j}\binom{m_G(i, j)}{2}\leq \frac{n^2}{a},$$
which gives
\begin{equation}\label{mij}
\sum_{i<j}m_G(i, j)^2\leq 8\frac{n^2}{a}.
\end{equation}

Let $B=\{(i,j): 1\leq i<j\leq n \text{ and } m_G(i, j)\neq 0\}$. By the definition of $m_G(i, j)$ and (\ref{mij}), we have $|B|\leq 8\frac{n^2}{a}/(\frac{\sqrt{n}}{b})^2=8b^2n/a.$ Therefore, by the convexity, we get
\begin{equation}\label{lbound}
\sum_{(i,j)\in B}m_G(i, j)^2\geq \frac{(\sum_{(i,j)\in B}m_G(i,j))^2}{|B|}=\frac{(\sum_{i<j}m_G(i,j))^2}{|B|}\geq\frac{(\sum_{i<j}m_G(i,j))^2}{8b^2n/a}.
\end{equation}
Combining (\ref{mij}) and (\ref{lbound}), we obtain $$\sum_{i<j}m_G(i,j)\leq \frac{8n^{3/2}b}{a}=\frac{4n^{3/2}}{\log^2 n}.$$
Finally, by the definition of $\widehat{G}$, we have $|E(G)-E(\widehat{G})|= \sum_{i<j}m_G(i,j)\leq\frac{4n^{3/2}}{\log^2 n}.$
\end{proof}
\begin{lemma}\label{cleanfamily}
Let $n$ be a sufficiently large integer. Then $|\mathcal{G}_n|\leq |\mathcal{\widehat{G}}_n|\cdot2^{\frac{4n^{3/2}}{\log n}}.$
\end{lemma}
\begin{proof}
For every $F\in \mathcal{\widehat{G}}_n$, let $\mathcal{S}_F=\{G\in\mathcal{G}_n \mid \widehat{G}=F\}$. By Lemma \ref{cleangraph}, for every $G\in \mathcal{S}_F$, we have $|E(G)-E(F)|\leq\frac{4n^{3/2}}{\log^2 n}$. Therefore, the size of $\mathcal{S}_F$ is bounded by
$$|\mathcal{S}_F|\leq \binom{\binom n2}{0}+\binom{\binom n2}{1}+\ldots+\binom{\binom n2}{\lfloor\frac{4n^{3/2}}{\log^2 n}\rfloor}\leq 2\binom{\binom n2}{\lfloor\frac{4n^{3/2}}{\log^2 n}\rfloor}\leq 2^{\frac{4n^{3/2}}{\log n}}.$$
Finally, we obtain that
$$|\mathcal{G}_n|\leq \sum_{F\in\mathcal{\widehat{G}}_n}|\mathcal{S}_F|\leq|\mathcal{\widehat{G}}_n|\cdot2^{\frac{4n^{3/2}}{\log n}}.$$
\end{proof}
\begin{thm}\label{mtc4}
Let $n$ be a sufficiently large integer. Then $|\mathcal{\widehat{G}}_n|\leq 2^{10n^{3/2}}$.
\end{thm}
\begin{proof}
We construct the certificate of a graph $G$ in the following way. Let $Y_G:= v_n<v_{n-1}<\ldots<v_1$ be the min-degree ordering of $G$ and $D_G:= \{d_n, d_{n-1}, \ldots, d_1\}$ be the min-degree sequence of $G$. Let $G_i=G[v_i,\ldots, v_1]$, for every $i\in[n]$.  Define the set sequence $S_G:= \{S_{n}, S_{n-1},\ldots, S_2\}$, where $S_{i}=N_G(v_i, G_{i-1})$. Then $S_i\subseteq \{v_{i-1},\ldots, v_1\}$, and $|S_i|=d_i$. By the construction,  $[Y_G, D_G, S_G]$ uniquely determines the graph $G$ and so we build a certificate $[Y_G, D_G, S_G]$ for $G$. Therefore, instead of counting graphs, it is equivalent to estimate the number of their certificates.

For a graph $G\in \mathcal{\widehat{G}}_n$, its certificate has some important properties which would help us to achieve the desired bound. 
First, by Corollary~\ref{degreecondition}, its min-degree ordering $D_G= \{d_n, d_{n-1}, \ldots,$ $d_1\}$ satisfying $d_i\leq 2\sqrt{n}$. 
Let $f_i$ be the number of 4-cycles  in $G_i$ containing vertex $v_i$. Since each 4-cycle contributes exactly to one of $f_i$'s, we have $\sum_{i=1}^nf_i\leq n^2/a.$ We call $v_i$ a \textit{heavy} vertex if $f_i>n/g$; otherwise, $v_i$ is a \textit{light} vertex.
Another crucial fact about graphs in $\mathcal{\widehat{G}}_n$ is that the number of heavy vertices is at most
\begin{equation}\label{class1}
\frac{\sum_{v_i\in V_h}f_i}{n/g} \leq \frac{n^2/a}{n/g}
=\frac{n}{\log n}.
\end{equation}

Now we start to estimate the number of certificates which would generate graphs in $\mathcal{\widehat{G}}_n$.
By the above discussion, we first observe that the number of ways to choose the min-degree orderings and the min-degree sequences is at most
\begin{equation}\label{minorder}
n!(2\sqrt{n})^{n}.
\end{equation}
Then we fix a min-degree ordering $Y^*= v_n<v_{n-1}<\ldots<v_1$, and a min-degree sequence $D^*= \{d_n, \ldots, d_1\}$. Next, we fix the positions of heavy vertices and by (\ref{class1}) the number of ways is at most 
\begin{equation}\label{heavyv}
\sum_{i\leq \frac{n}{\log n}}\binom{n}{i}.
\end{equation}
A major part of the proof is to count set sequences $S= \{S_{n}, S_{n-1},\ldots, S_2\}$, where $S_i\subseteq \{v_{i-1},\ldots, v_1\}$ and $|S_i|=d_i$, such that the graph reconstructed by $[Y^*, D^*, S]$, denoted by $G_S$, are in $\mathcal{\widehat{G}}$. 
For every $2\leq i\leq n$, let $M_i$ be the number of choices for $S_i$ with fixed sets $S_{i-1},$ $\ldots,$ $S_2$. 
Define
\[
\mathcal{I}_1=\{i: v_i \text{ is a heavy vertex}\},\quad \mathcal{I}_2=\{i: d_i<\frac{\sqrt{n}}{\log n}\},
\]
and
\[
 \quad \mathcal{I}_3=\{i: v_i \text{ is a light vertex and } d_i\geq \frac{\sqrt{n}}{\log n}\}.
\]
For every $i\in\mathcal{I}_1$, since $|S_i|=d_i\leq 2\sqrt{n}$, we have a trivial upper bound 
\begin{equation}\label{4cycleI1}
M_i\leq \binom{i-1}{d_i}\leq\binom{n}{2\sqrt{n}}\leq n^{2\sqrt{n}}=2^{2\sqrt{n}\log n}.
\end{equation}
Similarly, for every $i\in\mathcal{I}_2$, we have 
\begin{equation}\label{4cycleI2}
M_i\leq \binom{i-1}{d_i}\leq \binom{n}{\sqrt{n}/\log n}\leq n^{\sqrt{n}/\log n}=2^{\sqrt{n}}.
\end{equation}

It remains to estimate $M_i$ for $i\in\mathcal{I}_3$. 
With fixed sets $S_{i-1},$ $\ldots,$ $S_2$, the graph $G_{i-1}=G_S[v_{i-1},\ldots,v_1]$ is uniquely determined. Since $G_{i-1}\subseteq G_S$ and $G_S\in\mathcal{\widehat{G}}$, for every $u, v\in V(G_{i-1})$, we have $|N_{G_{i-1}}(u)\cap N_{G_{i-1}}(v)|\leq \sqrt{n}/b$. Applying Lemma $\ref{mainlemmaC4}$ on $G_{i-1}$, we obtain that every eligible $S_i$ contains a subset $T$ of size at most $2\sqrt{n}/\log n$, which determines a set $C(T)\supseteq S_i$ of size at most $5n/d_i$. Since the number of choices for $T$ is at most 
$$\sum_{0\leq j\leq 2\sqrt{n}/\log n}\binom{i-1}{j}\leq 2\binom{i-1}{2\sqrt{n}/\log n}\leq 2\binom{n}{2\sqrt{n}/\log n}\leq 2^{2\sqrt{n}}, $$
we then have
\begin{equation}\label{4cycleI3}
M_i\leq \sum_{T}\binom{C(T)}{d_i}\leq\sum_T\binom{5n/d_i}{d_i}
\leq \sum_T 2^{\frac{2}{2^{1/\ln 2}\ln 2}\sqrt{5en}}
\leq \sum_T 2^{4\sqrt{n}}
\leq 2^{6\sqrt{n}}
\end{equation}
for every $i\in\mathcal{I}_3$,
where the third inequality is given by Lemma~\ref{optimize}.

Combining (\ref{4cycleI1}), (\ref{4cycleI2}) and (\ref{4cycleI3}), we obtain that the number of choices for $S$ is 
\begin{equation*}
\prod_{i=2}^nM_i\leq \prod_{i\in\mathcal{I}_1}M_i\prod_{i\in\mathcal{I}_2}M_i\prod_{i\in\mathcal{I}_3}M_i
\leq (2^{2\sqrt{n}\log n})^{\frac{n}{\log n}}(2^{\sqrt{n}})^n(2^{6\sqrt{n}})^n
\leq 2^{9n^{3/2}}.
\end{equation*}
Finally, together with (\ref{minorder}) and (\ref{heavyv}), the total number of certificates is at most
$$n!(2\sqrt{n})^n\sum_{i\leq \frac{n}{\log n}}\binom{n}{i}\prod_{i=2}^nM_i\leq n!(2\sqrt{n})^n2^n2^{9n^{3/2}}\leq2^{10n^{3/2}},$$
which leads to $|\mathcal{\widehat{G}}_n|\leq 2^{10n^{3/2}}.$
\end{proof}

\noindent\textit{Proof of Theorem \ref{maintheoremC4}.} Lemma \ref{cleanfamily} and Theorem \ref{mtc4} imply Theorem \ref{maintheoremC4}.\qed
\section{The number of graphs with sparse short cycles}
In the previous section, we estimated the number of graphs containing a few 4-cycles. Unfortunately, we are not ready to provide a similar result for longer cycles due to the failure of getting an appropriate counting lemma, like Lemma~\ref{countinglemma}. However, this method still works when the target graph has a sparse structure on short cycles. More specially, for $\ell\geq 4$, we are going to consider the family of graphs such that each of its edges is contained in only $O(1)$ cycles of length at most $2\ell$.  Following the idea from~\cite{KKS}, we construct a proper auxiliary graph and provide a suitable counting lemma on it. 

\subsection{Expansion properties of graphs with sparse short cycles}
Given a graph $G$, a vertex $v\in V(G)$ and an integer $k\geq 1$, let $\Gamma_k(v)$ be the set of vertices of $G$ at distance exactly $k$ from $v$. Recall that for an edge $uv\in E(G)$, $c_{k}(u, v; G)$ is the number of $k$-cycles in $G$ containing edge $uv$.

\begin{lemma}\label{slevel}
For integers $\ell\leq m$ and a constant $L>0$, let $F$ be an $m$-vertex graph such that for every $uv\in E(F)$ and $3\leq i\leq 2\ell$, $c_i(u, v)\leq L$. Then for every $1\leq k\leq \ell-1$ and $v\in V(F)$, we have $$d(u, \Gamma_k(v))\leq Lk$$ for all $u\in \Gamma_k(v)$.
\end{lemma}
\begin{proof}
Suppose there exists a vertex $u\in\Gamma_k(v)$ such that $d(u, \Gamma_k(v))\geq Lk+1$. Since $u\in \Gamma_k(v)$, there exists a $(u, v)$-path $P_u$ of length $k$. Let $u'$ be the neighbor of $u$ in $P_u$. Similarly, for every vertex $w\in N(u, \Gamma_k(v))$, there is a $(w, v)$-path $P_w$ of length $k$. Note that every $P_u+P_w+\{uw\}$ forms a closed walk of length $2k+1$, which contains an odd cycle of length at most $2k+1$ containing edges $uu'$ and $uw$. Since $d(u, \Gamma_k(v))\geq Lk+1$, we have at least $Lk+1$ distinct odd cycles of length at most $2k+1$ containing $uu'$. However, since $c_h(u, u')\leq L$ for every odd $h\leq 2k+1$, there are at most $Lk$ odd cycles of length at most $2k+1$ containing $uu'$, which is a contradiction.
\end{proof}

\begin{lemma}\label{llevel}
For integers $\ell\leq m$ and a constant $L>0$, let $F$ be an $m$-vertex graph such that for every $uv\in E(F)$ and $3\leq i\leq 2\ell$, $c_i(u, v)\leq L$. Then for every $2\leq k\leq \ell$ and $v\in V(F)$, we have $$d(u, \Gamma_{k-1}(v))\leq L(k-1)+1$$ for all $u\in \Gamma_{k}(v)$. 
\end{lemma}
\begin{proof}
Suppose there exists a vertex $u\in\Gamma_k(v)$ such that $d(u, \Gamma_{k-1}(v))\geq L(k-1)+2$. Let $u'$ be a vertex in $N(u, \Gamma_{k-1}(v))$. Since $u'\in \Gamma_{k-1}(v)$, there exists a $(u', v)$-path $P_{u'}$ of length $k-1$. Similarly, for every vertex $w\in N(u, \Gamma_{k-1}(v))\setminus\{u'\}$, there is a $(w, v)$-path $P_w$ of length $k-1$. Note that every $P_{u'}+P_w+\{uu'\}+\{uw\}$ forms a closed walk of length $2k$, which contains an even cycle of length at most $2k$ containing edges $uu'$ and $uw$. Since $|N(u, \Gamma_{k-1}(v))\setminus\{u'\}|\geq L(k-1)+1$, we have at least $L(k-1)+1$ distinct even cycles of length at most $2k$ containing $uu'$. However, since $c_h(u, u')\leq L$ for every even $4\leq h\leq 2k$, there are at most $L(k-1)$ even cycles of length at most $2k$ containing $uu'$, which is a contradiction.
\end{proof}

Now, we give a lemma on the expansion of graphs with sparse short cycles. This lemma can be viewed as a generalization of Lemma 11 in~\cite{KKS}.

\begin{lemma}\label{exp}
For integers $\ell, d\leq m$ and a constant $L\ll d$, let $F$ be an $m$-vertex graph with minimum degree at least $d-1$, such that for every $uv\in E(F)$ and $3\leq i\leq 2\ell$, $c_i(u, v)\leq L$. Suppose $v$ is a vertex in $F$ with degree $d(v)$. Then for every $1\leq k\leq \ell$, we have
 $$|\Gamma_k(v)|\geq \frac{d(v)d^{k-1}}{g_k(L)}$$ 
 for some constants $g_k(L)$ which only depend on $k$ and $L$. 
\end{lemma}
\begin{proof}
The case $k=1$ is trivially true with $g_1(L)=1$. Suppose that the lemma is true for $k<\ell$, i.e.\ $|\Gamma_k(L)|\geq d(v)d^{k-1}/g_k(L)$ for some constant $g_k(L)$. 

For every vertex $u\in\Gamma_{k}(v)$, neighbors of $u$ only appear in $\Gamma_{k-1}(v)$, $\Gamma_{k}(v)$ and $\Gamma_{k+1}(v)$. By Lemmas~\ref{slevel} and \ref{llevel}, we have 
\begin{equation}\label{exdegree}
d(u, \Gamma_{k+1}(v))\geq (d-1)-d(u, \Gamma_{k-1}(v))-d(u, \Gamma_{k}(v))\geq d-2(Lk+1)+L\geq \frac d 2
\end{equation}
for all $u\in\Gamma_{k}(v)$ and this gives
$$e(\Gamma_{k}(v), \Gamma_{k+1}(v))\geq \frac{d|\Gamma_{k}(v)|}{2}.$$
Again by Lemma~\ref{llevel}, we know that for every $u\in \Gamma_{k+1}(v)$, $d(u, \Gamma_{k}(v))\leq Lk+1$. Therefore, we have
$$|\Gamma_{k+1}(v)|\geq \frac{e(\Gamma_{k}(v), \Gamma_{k+1}(v))}{Lk+1}\geq \frac{d|\Gamma_{k}(v)|}{2(Lk+1)}\geq \frac{d(v)d^{k}}{2(Lk+1)g_k(L)}=\frac{d(v)d^{k}}{g_{k+1}(L)}$$
for $g_{k+1}(L)=2(Lk+1)g_k(L)$ and the lemma follows by induction.
\end{proof}
Lemma~\ref{exp} gives an upper bound on the maximum degree of the graph with sparse short cycles.
\begin{cor}\label{mdegree}
For integers $\ell, d\leq m$ and a constant $L\ll d$, let $F$ be an $m$-vertex graph with minimum degree $d-1$, such that for every $uv\in E(F)$ and $3\leq i\leq 2\ell$, $c_i(u, v)\leq L$. Then $$\Delta(F)\leq \frac{m}{d^{\ell-1}}\cdot g_\ell(L),$$
where $g_{\ell}(L)$ is the constant defined in Lemma~\ref{exp}.
\end{cor}
\begin{proof}
By Lemma~\ref{exp}, for every $v\in V(F)$, we have
$$|\Gamma_{\ell}(v)|\geq \frac{d(v)d^{\ell-1}}{g_{\ell}(L)},$$ which gives 
$$d(v)\leq\frac{|\Gamma_{\ell}(v)|}{d^{\ell-1}}g_{\ell}(L)\leq \frac{m}{d^{\ell-1}}g_{\ell}(L),$$
This implies the corollary.
\end{proof}

\subsection{Construction of the auxiliary graph}
In this section, we aim to give a generalization of Lemma~\ref{countinglemma} for longer cycles. We use a definition of \textit{composed walk} from \cite{KKS}. For every integer $k\geq 1$, call a $2k$-walk $x_0x_1\ldots x_{2k}$ a \textit{composed walk} if $x_0\ldots x_{k}$ and $x_{k}\ldots x_{2k}$ are two shortest paths and they are different but not necessarily vertex-disjoint or edge-disjoint. A composed walk is said to be \textit{closed} if its endpoints are the same.

\begin{lemma}\label{cwalk}
For integers $\ell, \Delta\leq m$ and a constant $L\ll \Delta$, let $F$ be an $m$-vertex graph with maximum degree $\Delta$, such that for every $uv\in E(F)$ and $3\leq k\leq 2\ell$, $c_k(u, v)\leq L$. Then for every vertex $u\in V(F)$ and every integer $2\leq s\leq \ell-1$, the number of closed composed walks of length $2s$ with endpoints $u$ is at most $$\Delta^{s-1} \alpha_s(L)$$ for some constants $\alpha_s(L)$ which only depends on $s$ and $L$.
\end{lemma}

\begin{proof}
For every vertex $u\in V(F)$ and every integer $2\leq s\leq \ell-1$, let $\mathcal{W}_s(u)$ be the set of closed composed walks of length $2s$ with endpoints $u$.
For the case $s=2$, the lemma is true with $\alpha_2(L)=L$. This is because that a closed composed walk of length $4$ with endpoint $u$ is exactly a 4-cycle containing $u$ and then we have $|\mathcal{W}_2(u)|\leq \sum_{v\in N(u)}c_4(u, v)\leq \Delta L$.

Suppose for $s-1< \ell-1$, the lemma is true for all integers $k\leq s-1$, i.e.\ for every $v\in V(F)$, $|\mathcal{W}_k(v)|\leq \Delta^{k-1} \alpha_k(L)$ with some constants $\alpha_k(L)$. Fix an arbitrary vertex $u\in V(F)$,  and let 
$$\mathcal{W}^i_s(u)=\{ux_1x_2\ldots x_{2s-1}u\in\mathcal{W}_s(u) \mid i\textrm{ is the first integer such that } x_{i}=x_{2s-i}\}$$ 
for every $1\leq i\leq s$. Then we have $\mathcal{W}_s(u)=\bigcup_{i=1}^s\mathcal{W}^i_s(u).$

First, every composed walk $W\in \mathcal{W}^1_s(u)$ consists of an edge $ux_1$ and a closed composed walk of length $2s-2$ with endpoints $x_1$. Therefore, we have
$$|\mathcal{W}^1_s(u)|\leq \sum_{x_i\in N(u)}|\mathcal{W}_{s-1}(x_1)|\leq \Delta^{s-1} \alpha_{s-1}(L).$$
Let $2\leq i\leq s-1$. For every composed walk 
$$W=ux_1x_2\ldots x_{2s-1}u\in \mathcal{W}^i_s(u),$$
$\{ux_1\ldots x_i$ $x_{2s-(i-1)}\ldots x_{2s-1}u\}$ forms a cycle $C$ of length $2i$ containing $u$. Since for every $x_1\in N(u)$, $c_{2i}(u, x_1)\leq L$, then the number of choices for $C$ is at most $\Delta L$. For a fixed $C$ and $x_i\in C$, $W-C$ forms a path of length $(s-i)$ with endpoints $x_i$ or a closed composed walks of length $2(s-i)$ with endpoints $x_i$. In the first case there are at most $\Delta^{s-i}$ choices, while in the later case there are at most $|\mathcal{W}_{s-i}(x_i)|$ choices. 
Therefore, we have
\begin{equation*}\begin{split}
|\mathcal{W}^i_s(u)|&\leq \Delta L\cdot (\Delta^{s-i}+|\mathcal{W}_{s-i}(x_i)|)\\
&\leq \Delta L\cdot (\Delta^{s-i}+ \Delta^{s-i-1} \alpha_{s-i}(L))\\
&\leq 2\Delta^{s-i+1}L\leq 2\Delta^{s-1}L.
\end{split}\end{equation*}
Finally, every composed walk $W\in \mathcal{W}^s_s(u)$ is a cycle of length $2s$ containing $u$, and then we have 
$$|\mathcal{W}^s_s(u)|\leq \sum_{v\in N(u)}c_{2s}(u, v)\leq \Delta L.$$
Hence, we have 
\begin{equation*}\begin{split}
|\mathcal{W}_s(u)|&=\bigcup_{i=1}^{s}|\mathcal{W}^i_s(u)|\leq\Delta^{s-1} \alpha_{s-1}(L)+2(s-2)\Delta^{s-1}L+\Delta L\leq \Delta^{s-1} \alpha_{s}(L)
\end{split}\end{equation*}
for $\alpha_{s}(L)=\alpha_{s-1}(L)+2(s-2)L+1,$ and the lemma follows by induction.
\end{proof}
For an integer $\ell\geq 3$ and a graph $F$, denote by $F^{\ell}$ the multigraph defined on $V(F)$ such that for every distinct $u, v\in V(F^{\ell})$, the multiplicity of $uv$ in $F^{\ell}$ is the number of composed $(u, v)$-walks of length $2(\ell-1)$ in $F$. 

\begin{lemma}\label{fwalk}
For an integer $\ell\geq 3$ and a constant $L>0$, let $n$ be a sufficiently large integer. Let $m$ and $d$ be the integers satisfying $m\leq n$ and $d\geq \frac{n^{1/\ell}}{\log n}$. Suppose $F$ is an $m$-vertex graph with minimum degree $d-1$, such that for every $uv\in E(F)$ and $3\leq k\leq 2\ell$, $c_k(u, v)\leq L$.  Then for every set $J\subseteq V(F)$ of size at least $2^{\ell}n/d^{\ell-1}$, we have $$e(F^{\ell}[J])\geq\frac{d^{2\ell-2}|J|^2}{2^{2\ell+1}n} .$$
\end{lemma}

\begin{proof}
Let $\mathcal{W}$ be the set of composed walks of length $2(\ell-1)$ with endpoints in $J$, and $\mathcal{W}_c$ be the set of closed composed walks of length $2(\ell-1)$ with endpoint in $J$. By the definition of $F^{\ell}$, we have $$e(F^{\ell}[J])=|\mathcal{W}|-|\mathcal{W}_c|.$$ By Lemma~\ref{cwalk}, we know that
$$|\mathcal{W}_c|\leq \Delta^{\ell-2}\alpha_{\ell-1}(L)\cdot|J|,$$
where $\Delta$ is the maximum degree of $F$, which, by Corollary~\ref{mdegree}, satisfies
\begin{equation}\label{umdegree}
\Delta\leq \frac{m}{d^{\ell-1}}\cdot g_{\ell}(L)\leq \frac{n}{d^{\ell-1}}\cdot g_{\ell}(L)\leq\frac{d^{\ell}\log^{\ell} n}{d^{\ell-1}}\cdot g_{\ell}(L)=d\log^{\ell} n\cdot g_{\ell}(L).
\end{equation}

Now, it remains to estimate the lower bound of $\mathcal{W}$. For every $v\in J$, let $a_v$ be the number of shortest paths of length $\ell-1$ such that $v$ is one of the endpoints. For every $u\in V(F)$, let $\mathcal{P}_u$ be the set of shortest paths of length $\ell-1$ such that one endpoint is $u$ and another endpoint is in $J$. Let $b_u=|\mathcal{P}_u|$ and then we have $\sum_{u\in V(F)}b_u=\sum_{v\in J}a_v.$
By (\ref{exdegree}), we have $$\sum_{u\in V(F)}b_u=\sum_{v\in J}a_v\geq\sum_{v\in J}(d/2)^{\ell-1}=\frac{d^{\ell-1}|J|}{2^{\ell-1}}.$$
Note that for every vertex $u\in V(F)$ and $P_1, P_2\in \mathcal{P}_u$, $P_1+P_2$ forms a composed walk in $\mathcal{W}$ and vice versa. Therefore, we have
$$|\mathcal{W}|=\sum_{u\in V(F)}\binom{b_u}{2}\geq m\binom{\frac{\sum_{u}b_u}{m}}{2}\geq m\binom{\frac{d^{\ell-1}|J|}{2^{\ell-1}\cdot m}}{2}\geq \frac{d^{2\ell-2}|J|^2}{2^{2\ell}m}\geq\frac{d^{2\ell-2}|J|^2}{2^{2\ell}n}$$
for $|J|\geq 2^{\ell}n/d^{\ell-1}\geq2^{\ell}m/d^{\ell-1}.$
Note that
\begin{equation*}
\begin{split}
\frac{|\mathcal{W}_c|}{|\mathcal{W}|}&\leq \frac{\Delta^{\ell-2}\alpha_{\ell-1}(L)\cdot|J|}{\frac{d^{2\ell-2}|J|^2}{2^{2\ell}n}}
\leq \frac{d^{\ell-2}\log^{\ell(\ell-2)} n}{d^{2\ell-2}}\cdot\frac{n}{|J|}\cdot g_{\ell}^{\ell-2}(L)\alpha_{\ell-1}(L)2^{2\ell}\\
&\leq \frac{d^{\ell-2}\log^{\ell(\ell-2)} n}{d^{2\ell-2}}\cdot\frac{d^{\ell-1}}{2^\ell}\cdot g_{\ell}^{\ell-2}(L)\alpha_{\ell-1}(L)2^{2\ell}\\
&\leq \frac{\log^{\ell(\ell-2)} n}{d}\cdot g_{\ell}^{\ell-2}(L)\alpha_{\ell-1}(L)2^{\ell}\ll 1,
\end{split}
\end{equation*}
when $n$ is sufficiently large.
Hence, we have $$e(F^{\ell}[J])=|\mathcal{W}|-|\mathcal{W}_c|\geq \frac12|\mathcal{W}|\geq \frac{d^{2\ell-2}|J|^2}{2^{2\ell+1}n}.$$
\end{proof}

Now, we start to define the auxiliary graph, which will be used in Lemma~\ref{mainlemmalonger} in the next section. For every integer $k\geq 1$, call a path $x_0x_1\ldots x_{2k}$ a \textit{composed path} if $x_0\ldots x_{k}$ and $x_{k}\ldots x_{2k}$ are both shortest paths of length $k$. For an integer $\ell\geq 3$ and a graph $F$, denote by $F^{\ell}_*$ the simple graph defined on $V(F)$ such that for every distinct $u, v\in V(F^{\ell}_*)$, $uv\in E(F^{\ell}_*)$ if there is a composed $(u, v)$-path of length at most $2(\ell-1)$ in $F$. To estimate the number of edges in $F^{\ell}_*$, we need the following lemma.

\begin{lemma}\label{cpath}
For integers $\ell, \Delta\leq m$ and a constant $L\ll \Delta$, let $F$ be an $m$-vertex graph with maximum degree $\Delta$, such that for every $uv\in E(F)$ and $3\leq k\leq 2\ell$, $c_k(u, v)\leq L$. For every $1\leq s\leq \ell-1$ and every distinct $u, v\in V(F)$, the number of composed paths of length $2s$ with endpoints $u, v$ is at most
$$\Delta^{s-1}\left((sL+1)^s+1\right).$$
\end{lemma}

\begin{proof}
Let $\mathcal{P}$ be the set of composed paths of length $2s$ with endpoints $u, v$. For given vertices $a_1,\ldots, a_{s-1}$, let $$\mathcal{P}(a_1,\ldots,a_{s-1})=\{ux_1\ldots x_{2s-1}v\in\mathcal{P}\mid x_1=a_1, \ldots, x_{s-1}=a_{s-1}\}.$$ Note that the number of non-empty $\mathcal{P}(a_1,\ldots,a_{s-1})$ is at most $\Delta^{s-1}$, since $ua_1\ldots a_{s-1}$ is a path.

Suppose that $P_0=ua_1\ldots a_{2s-1}v$ is a composed path in $\mathcal{P}(a_1,\ldots,a_{s-1})$. For every composed path 
$P=ua_1\ldots a_{s-1}x_s\ldots x_{2s-1}v\in \mathcal{P}(a_1,\ldots,a_{s-1})\setminus\{P_0\},$ $a_{s-1}\ldots a_{2s-1}v$ and $a_{s-1}x_s\ldots x_{2s-1}v$ form a closed walk $W$ of length $2(s+1)$. For every $s\leq i\leq 2s-1$, if $x_i=a_i$, the number of choices for $x_i$ is 1. Otherwise, $W$ contains an even cycle of length at most $2(s+1)$, which contains the edge $a_{i-1}a_i$ and vertex $x_i$. Since $c_{2k}(a_{i-1}, a_i)\leq L$ for every $2\leq k\leq s+1$, the number of choices for $x_i\neq a_i$ is at most $sL$. Therefore, we have 
$$|\mathcal{P}(a_1,\ldots,a_{s-1})|\leq (sL+1)^s+1.$$
Finally, we obtain
$$|\mathcal{P}|=\sum_{a_1,\ldots, a_{s-1}}|\mathcal{P}(a_1,\ldots,a_{s-1})|\leq \Delta^{s-1}\left((sL+1)^s+1\right).$$
\end{proof}

Now, we give an upper bound on the multiplicity of $F^{\ell}$.

\begin{lemma}\label{multi}
For integers $\ell, \Delta\leq m$ and a constant $L\ll \Delta$, let $F$ be an $m$-vertex graph with maximum degree $\Delta$, such that for every $uv\in E(F)$ and $3\leq k\leq 2\ell$, $c_k(u, v)\leq L$. For every distinct $u, v\in V(F)$, the number of composed walks of length $2(\ell-1)$ with endpoints $u, v$ is at most
$$\Delta^{\ell-2}\beta_{\ell}(L),$$
for a constant $\beta_{\ell}(L)$ which only depends on $\ell$ and $L$.
\end{lemma}

\begin{proof}
Let $\mathcal{W}$ be the number of composed walks of length $2(\ell-1)$ in $F$ with endpoints $u, v$. For every $1\leq i\leq \ell-1$, let $$\mathcal{W}_i=\{ux_1\ldots x_{2(\ell-1)-1}v\in \mathcal{W} \mid i \textrm{ is the first integer such that }x_i=x_{2(\ell-1)-i}\},$$ and then we have $\mathcal{W}=\bigcup_{i=1}^{\ell-1}\mathcal{W}_i$.  
 
Let $1\leq i\leq \ell-2$. For every composed walk $$W=ux_1\ldots x_{\ell-1}\ldots x_{2(\ell-1)-1}v\in \mathcal{W}_i,$$  $\{ux_1\ldots x_ix_{2(\ell-1)-(i-1)}\ldots x_{2(\ell-1)-1}v\}$ forms a composed path $P$ of length $2i$. By Lemma~\ref{cpath}, there are at most $\Delta^{i-1}[(iL+1)^i+1]$ choices for $P$. For a fixed $P$, $W-P$ forms a path of length $\ell-i-1$ with endpoint $x_i$ or a close composed walk of length $2(\ell-i-1)$ with endpoint $x_i$. In the first case, there are at most $\Delta^{\ell-i-1}$ choices, while in the later case, by Lemma~\ref{cwalk}, there are at most $\Delta^{\ell-i-2}\alpha_{\ell-i-1}(L)$ choices. Therefore, we have
\begin{equation*}
\begin{split}
|\mathcal{W}_i|&\leq \Delta^{i-1}\left((iL+1)^i+1\right)\cdot (\Delta^{\ell-i-1}+\Delta^{\ell-i-2}\alpha_{\ell-i-1}(L))
\leq 2\Delta^{\ell-2}\left((iL+1)^i+1\right).
\end{split}
\end{equation*}
Moreover, every walk $W\in \mathcal{W}_{\ell-1}$ is a composed path of length $2(\ell-1)$ with endpoints $u$ and $v$. By Lemma~\ref{cpath}, we have $$|\mathcal{W}_{\ell-1}|\leq \Delta^{\ell-2}\left((\ell L-L+1)^{\ell-1}+1\right).$$
Hence, we have
\begin{equation*}
\begin{split}
|\mathcal{W}|&=\sum_{i=1}^{\ell-1}|\mathcal{W}_i|
\leq \sum_{i=1}^{\ell-2}2\Delta^{\ell-2}\left((iL+1)^i+1\right)+\Delta^{\ell-2}\left((\ell L-L+1)^{\ell-1}+1\right)
=\Delta^{\ell-2}\beta_{\ell}(L)
\end{split}
\end{equation*}
for $\beta_{\ell}(L)=\sum_{i=1}^{\ell-2}2\left((iL+1)^i+1\right)+\left((\ell L-L+1)^{\ell-1}+1\right).$
\end{proof}
We have all the ingredients to give a lower bound on the number of edges in auxiliary graph $F^l_*$. This lemma will play the same role as Lemma~\ref{countinglemma} in the case of 4-cycles.

\begin{lemma}\label{fpath}
For an integer $\ell\geq 3$ and a constant $L>0$, let $n$ be a sufficiently large integer. Let $m$ and $d$ be the integers satisfying $m\leq n$ and $d\geq\frac{n^{1/\ell}}{\log n}$. Suppose $F$ is an $m$-vertex graph with minimum degree $d-1$, such that for every $uv\in E(F)$ and $3\leq k\leq 2\ell$, $c_k(u, v)\leq L$.  Then for every set $J\subseteq V(F)$ of size at least $2^{\ell}n/d^{\ell-1}$, we have $$e(F^{\ell}_*[J])\geq\frac{d^{\ell}|J|^2}{n\log^{\ell(\ell-2)} n}f_{\ell}(L)$$
for a constant $f_{\ell}(L)$ which only depends on $\ell$ and $L$.
\end{lemma}

\begin{proof}
Note that every composed walk of length $2(\ell-1)$ with endpoints in $J$ contains a composed path of length at most $2(\ell-1)$ with endpoints in $J$. Therefore, by Lemma~\ref{multi}, we have 
$$e(F^{\ell}_*[J])\geq \frac{e(F^{\ell}[J])}{\Delta^{\ell-2}\beta_{\ell}(L)},$$
where $\Delta$ is the maximum degree of $F$, which by (\ref{umdegree}), satisfies
$$\Delta\leq d\log^{\ell} n\cdot g_{\ell}(L).$$
Hence, we have 
$$e(F^{\ell}_*[J])\geq \frac{e(F^{\ell}[J])}{d^{\ell-2}\log^{\ell(\ell-2)} n\cdot g_{\ell}^{\ell-2}(L)\beta_{\ell}(L)}\geq \frac{d^{\ell}|J|^2}{n\log^{\ell(\ell-2)} n}f_{\ell}(L),$$
where $f_{\ell}(L)=\frac{1}{2^{2\ell+1}g_{\ell}^{\ell-2}(L)}\beta_{\ell}(L).$
\end{proof}

\subsection{Certificate lemma}
In this section, we give our second main lemma, which will be used to build certificates for graphs with sparse short cycles. This lemma is a generalization of Lemma~\ref{mainlemmaC4} for longer cycle, although the condition is slightly different. The idea of proof is also similar to Lemma~\ref{mainlemmaC4}, which originally comes from Kleitman and Winston~\cite{kleitmanwinston} and Kohayakawa, Kreuter and Steger~\cite{KKS}.

\begin{lemma}\label{mainlemmalonger}
For an integer $\ell\geq 3$ and constants $L,\alpha>0$, let $n$ be a sufficiently large integer. Let $m$ and $d$ be the integers satisfying $m\leq n$ and $\frac{n^{1/\ell}}{\log n}\leq d\leq \alpha n^{1/\ell}$. Suppose $F$ is an $m$-vertex graph with minimum degree $d-1$, such that for every $uv\in E(F)$ and $3\leq k\leq 2\ell$, $c_k(u, v)\leq L$.  Let $H=F^{\ell}_*$. Then for every set $I\subseteq V(F)$ of size $d$, such that $d_{H}(v, I)\leq (\ell-1)L$ for all $v\in I$, there exist a set $T$ and a set $C(T)$ depending only on $T$, not on $I$, such that
\begin{enumerate}[label={\upshape(\roman*)}]
  \item $T\subseteq I\subseteq C(T)$,
  \item $|T|\leq n^{1/\ell}/\log n$,
  \item $|C(T)|\leq (2^{\ell}+1)n/d^{\ell-1}$.
\end{enumerate}
\end{lemma}
\begin{proof}
This proof is similar to the proof of Lemma~\ref{mainlemmaC4}. We will describe a deterministic algorithm that associates to the set $I$ a pair of sets $T$ and $C(T)$.

We start the algorithm with sets $A_0=V(H)$, $T_0=\emptyset$ and a function $t_0(v)=0$, for every $v\in V(H)$. In the $i$-th iteration step, we pick a vertex $u_i\in A_i$ of maximum degree in $H[A_i]$. 
If $u_i\in I$, we define
  $$t_{i+1}(v)=\left\{\begin{array}{ll} 
    t_i(v)+d_H(v, u_i) & \text{if } v\in A_{i},\\
    t_i(v)  & \text{if } v\notin A_{i},
    \end{array}\right.$$
and $Q=\{v \mid t_{i+1}(v)>(\ell-1)L\}$, and let $T_{i+1}=T_i+u_i$, $A_{i+1}=A_i-u_i-Q$.
Otherwise, let $T_{i+1}=T_i$, $A_{i+1}=A_i-u_i$ and $t_{i+1}(v)=t_i(v)$, for every $v\in V(H)$. 
The algorithm terminates at step $K$ when we get a set $A_{K}$ of size at most $2^{\ell}n/d^{\ell-1}$. We also assume that $u_{K-1}\in T_K$ as otherwise we can continue the algorithm until it is satisfied.\\

The algorithm outputs a vertex sequence $\{u_1, u_2, \ldots, u_{K-1}\}$, a set of `representative' vertices $T_K$ and a strictly decreasing set sequence $\{A_0, A_1, A_2, A_3, \ldots A_K\}$.  Let $T=T_{K}$ and $C(T)=A_{K}\cup T.$ From the algorithm, we have $T\subseteq I$. Furthermore, if a vertex $v$ satisfies $t_{i}(v)>(\ell-1)L$ for some $i$, then we have $d_H(v, I)\geq t_i(v)>(\ell-1)L$, which implies $v\notin I$. Therefore, we maintain  $I\subseteq A_i\cup T_i$ for every $i\leq K$ and especially get $I\subseteq A_{K}\cup T_{K} = C(T)$. Hence, Condition (i) is satisfied. Similarly as in Lemma~\ref{mainlemmaC4}, the set $C(T)$ only depends on $T$, not on $I$.

To finish the proof, it is sufficient to show that $|T_{K}|\leq n^{1/\ell}/\log n$. Once we prove it, we immediately obtain $$|T|=|T_{K}|\leq \frac{n^{1/\ell}}{\log n},$$
and $$|C(T)|=|A_{K}|+|T|\leq \frac{2^{\ell}n}{d^{\ell-1}}+\frac{n^{1/\ell}}{\log n}\leq \frac{(2^{\ell}+1)n}{d^{\ell-1}},$$ which  completes the proof.

In the rest of proof, we apply the same technique used in the proof of Lemma~\ref{mainlemmaC4}. We repeat the process as follows.
Denote $q$ the integer such that $n/2^{q}\leq|A_{K}|< n/2^{q-1}$. By the choice of $A_K$, we have $q\leq \log n$. For every integer $1\leq l\leq q$, define $A^l$ to be the first $A$-set satisfying 
\[
\frac{n}{2^{l}}\leq|A^l|< \frac{n}{2^{l-1}},
\] if it exists, and let $T^l$ be the corresponding $T$-set and $t^l(v)$ be the corresponding $t$-function of $A^l$. Note that $A^l$ may not exist for every $l$, but $A^{q}$ always exists and it could be that $A^q=A_{K}.$ Suppose $$A^{l_1}\supset A^{l_2}\supset\ldots\supset A^{l_{p}}$$ are all the defined $A^l$, where $p\leq q$. By the above definition, we have $A^{l_1}=A_0$, $T^{l_1}=T_0$ and $l_p=q$. Define $A^{l_{p+1}}=A_{K}, T^{l_{p+1}}=T_{K}$. Now, we have 
\begin{equation}\label{TKlonger}
T_{K}=\bigcup_{j=2}^{p+1}(T^{l_j}-T^{l_{j-1}}).
\end{equation} To achieve our goal, we are going to estimate the size of $T^{l_j}-T^{l_{j-1}}$, for every $2\leq j\leq p+1$.

From the algorithm, we have $t^{l_j}(v)\leq (\ell-1)L$, for every $v\in A^{l_j}\cup T^{l_j}$. Moreover, for $v\in A^{l_{j-1}}-A^{l_j}-T^{l_j}$, suppose $v$ is removed in step $i$, then we have 
$$t^{l_j}(v)= t_{i-1}(v)+d_H(v, u_{i-1})\leq(\ell-1)L+1,$$
where $u_i$ is the selected vertex in step $i$.
Therefore, we obtain
\begin{equation}\label{ublong}
\sum_{v\in A^{l_{j-1}}}t^{l_j}(v)\leq\left((\ell-1)L+1\right)|A^{l_{j-1}}|\leq \left((\ell-1)L+1\right)\frac{n}{2^{l_{j-1}-1}}. 
\end{equation}

Let $2\leq j\leq p$. For every $u_i\in T^{l_j}-T^{l_{j-1}}$, $u_i$ is chosen of maximum degree in $H[A_i]$, where $A_i$ is a set between $A^{l_{j-1}}$ and $A^{l_j}$.  By the choice of $A^{l_j}$, we have $|A_i|\geq n/2^{l_{j-1}}$. From Lemma $\ref{fpath}$, we obtain that
$$d_H(u_i, A_i)\geq \frac{d^{\ell}|A_i|}{n\log^{\ell(\ell-2)} n}f_{\ell}(L)\geq \frac{d^{\ell}}{2^{l_{j-1}}\log^{\ell(\ell-2)} n}f_{\ell}(L).$$
 Note that $d_H(u_i, A_i)$ only contributes to $t^{l_j}(v)$, for $v\in A_i\subseteq A^{l_{j-1}}$. Then we obtain 
 \begin{equation}\label{lblong}
 \left|T^{l_j}-T^{l_{j-1}}\right|\frac{d^{\ell}}{2^{l_{j-1}}\log^{\ell(\ell-2)} n}f_{\ell}(L)\leq\sum_{u_i\in T^{l_j}-T^{l_{j-1}}}d_H(u_i, A_i)\leq\sum_{v\in A^{l_{j-1}}}t^{l_j}(v).
\end{equation}
Combining (\ref{ublong}) and (\ref{lblong}), we have
\begin{equation*}
 \left|T^{l_j}-T^{l_{j-1}}\right|\frac{d^{\ell}}{2^{l_{j-1}}\log^{\ell(\ell-2)} n}f_{\ell}(L)\leq\left((\ell-1)L+1\right)\frac{n}{2^{l_{j-1}-1}},
\end{equation*}
which implies 
\begin{equation*}
\begin{split}
\left|T^{l_j}-T^{l_{j-1}}\right|
&\leq \frac{2\left(\ell-1)L+1\right)}{f_{\ell}(L)}\cdot\frac{n\log^{\ell(\ell-2)} n}{d^{\ell}}
\leq \frac{2\left((\ell-1)L+1\right)}{f_{\ell}(L)}\log^{\ell(\ell-1)} n
\leq \frac{n^{1/\ell}}{\log^2 n},
\end{split}
\end{equation*}
for $2\leq j\leq p.$
For $j=p+1$, since we have $\frac{n}{2^q}\leq |A^{l_{p+1}}|\leq|A^{l_p}|\leq\frac{n}{2^{q-1}}$, by a similar argument, we obtain that
$$\left|T^{l_{p+1}}-T^{l_{p}}\right|\frac{d^{\ell}|A_{l_{p+1}}|}{n\log^{\ell(\ell-2)} n}f_{\ell}(L)\leq \sum_{u_i\in T^{l_{p+1}}-T^{l_{p}}}d(u_i, A_i)\leq \sum_{v\in A^{l_{p}}}t^{l_{p+1}}(v)\leq\left((\ell-1)L+1\right)|A^{l_{p}}|,$$
which gives 
$$\left|T^{l_{p+1}}-T^{l_{p}}\right|\leq \frac{2\left((\ell-1)L+1\right)}{f_{\ell}(L)}\cdot\frac{n\log^{\ell(\ell-2)} n}{d^{\ell}}\leq\frac{n^{1/\ell}}{\log^2 n}.$$
Finally, by (\ref{TKlonger}), we have $$|T_{K}|=\bigcup_{j=2}^{p+1}|T^{l_j}-T^{l_{j-1}}|\leq p\cdot\frac{n^{1/\ell}}{\log^2 n}\leq q\cdot\frac{n^{1/\ell}}{\log^2 n}\leq\frac{n^{1/\ell}}{\log n}.$$
\end{proof}

\subsection{Proof of Theorem~\ref{mainthmlonger}}
This section is entirely devoted to the proof of Theorem~\ref{mainthmlonger}. The idea is the same as the proof of Theorem~\ref{mtc4}: we will build a certificate for each graph in $\mathcal{G}_n(2\ell, L)$ and estimate the number of such certificates. Before we proceed, we first need the supersaturation result for $C_{2\ell}$ to give a bound on the min-degree sequence of graphs in $\mathcal{G}_n(2\ell, L)$.  It was mentioned in \cite{ES} that Simonovits first proved the supersaturation for the even cycles, but the proof has not been published yet and it might appear in an upcoming paper of Faudree and Simonovits~\cite{RM}. Morris and Saxton~\cite{MS} recently provided a stronger version of supersaturation for even cycles. Very recently, Jiang and Yepremyan~\cite{JY} give a supersaturation result of even linear cycles in linear hypergraphs, which includes the graph case. We use the graph version of their result and rephrase it in terms of the average degree.
\begin{thm}\cite{JY}\label{super}
For an integer $\ell\geq 2$, there exist constants $C, c$ such that if $G$ is an $n$-vertex graph with the average degree $d\geq 2Cn^{1/\ell}$, then $G$ contains at least $c(\frac{d}{2})^{2\ell}$ copies of $C_{2\ell}$.
\end{thm}

\begin{cor}\label{degreelonger}
Let $G$ be a $n$-vertex graph in $\mathcal{G}_n(2\ell, L)$, and $d_n, \ldots, d_1$ be the min-degree sequence of $G$. Then for every $i\in [n]$, we have $$d_i\leq \alpha n^{1/\ell}$$
for some constant $\alpha=\max\{2C, 2(\frac{L}{2c})^{1/2\ell}\}$, where $C$, $c$ are constants given in Theorem~\ref{super}.
\end{cor}

\begin{proof}
Suppose that there exists $k\in [n]$, such that $d_k>\alpha n^{1/\ell}.$
Then by Theorem~\ref{super}, the number of $C_{2\ell}$'s in $G_k$ is at least 
\[
c\left(\frac{d_k}{2}\right)^{2\ell}>c\left(\frac{\alpha n^{1/\ell}}{2}\right)^{2\ell}\geq c\frac{L}{2c}n^2\geq L\binom{k}{2},
\]
which contradicts the fact that $G\in \mathcal{G}_n(2\ell, L)$.
\end{proof}
 
\noindent\textit{Proof of Theorem~\ref{mainthmlonger}.}
The way to construct the certificate is exactly same with in the proof of Theorem~\ref{mtc4}. Here we restate the process.
For a graph $G\in \mathcal{G}_n(2\ell, L)$, let $Y_G:= v_n<v_{n-1}<\ldots<v_1$ be the min-degree ordering of $G$ and $D_G:= \{d_n, d_{n-1}, \ldots, d_1\}$ be the min-degree sequence of $G$. Note that by Corollary~\ref{degreelonger}, there exists a constant $\alpha$ such that $d_i\leq \alpha n^{1/\ell}$, for every $i\in[n]$. For every $i\in[n]$, let $G_i=G[v_i,\ldots, v_1]$. Define the set sequence $S_G:= \{S_{n}, S_{n-1},\ldots, S_2\}$, where $S_{i}=N_G(v_i, G_{i-1})$. Note that $S_i\subseteq \{v_{i-1},\ldots, v_1\}$ and $|S_i|=d_i$. By the construction,  $[Y_G, D_G, S_G]$ uniquely determines the graph $G$ and so we build a certificate $[Y_G, D_G, S_G]$ for $G$. To complete the proof, it is sufficient to estimate the number of such certificates.
 
We first choose a min-degree ordering $Y^*=v_n<v_{n-1}<\ldots<v_1$, and a min-degree sequence $D^*=\{d_n, \ldots, d_1\}$; the number of options is at most 
\begin{equation}\label{girth: minorder}
n!(\alpha n^{1/\ell})^n.
\end{equation}
Next, we count set sequences $S=\{S_{n}, S_{n-1},\ldots, S_2\}$, where $S_i\subseteq \{v_{i-1},\ldots, v_1\}$ and $|S_i|=d_i$, such that the graph reconstructed by $[Y^*, D^*, S]$, denoted by $G_S$, are in $\mathcal{G}_n(2\ell, L)$. 
For every $2\leq i\leq n$, let $M_i$ be the number of choices for $S_i$ with fixed sets $S_{i-1},\ldots, S_2$. 
Define
\[
\mathcal{I}_1=\{i: d_i<n^{1/\ell}/\log n\}, \quad \quad \mathcal{I}_2=\{i: d_i\geq n^{1/\ell}/\log n\}.
\]
For every $i\in \mathcal{I}_1$, since $|S_i|=d_i<n^{1/\ell}/\log n$, we have a trivial bound
\begin{equation}\label{lcycleI1}
M_i\leq \binom{i-1}{d_i}\leq \binom{n}{n^{1/\ell}/\log n}\leq n^{n^{1/\ell}/\log n}=2^{n^{1/\ell}}.
\end{equation}

It remains to consider the upper bound on $M_i$ for $i\in\mathcal{I}_2$. 
With fixed sets $S_{i-1},\ldots, S_2$, the graph $G_{i-1}=G_S[v_{i-1},\ldots,v_1]$ is uniquely determined. Since $G_{i-1}\subseteq G_S$ and $G_S\in\mathcal{G}_n(2\ell, L)$, for every $uv\in E(G_{i-1})$ and every $3\leq k\leq 2\ell$, we know that $c_k(u, v; G_{i-1})\leq L$. Note that every eligible $S_i$ should satisfy $d_H(u, S_i)\leq (l-1)L$ for all $u\in S_i$, where $H=(G_{i-1})^l_*.$ Otherwise, there exists a vertex $u\in S_i$, such that $\sum_{k=2}^{\ell}c_{2k}(v_i, u; G_i)\geq d_H(u, S_i)>(l-1)L$, which is a contradiction.
Applying Lemma $\ref{mainlemmalonger}$ on $G_{i-1}$, we obtain that every eligible $S_i$ contains a subset $T$ of size at most $n^{1/\ell}/\log n$, which uniquely determines a set $C(T)\supseteq S_i$ of size at most $(2^{\ell}+1)n/d_i^{\ell-1}$. Since the number of choices for $T$ is at most 
$$\sum_{0\leq j\leq n^{1/\ell}/\log n}\binom{i-1}{j}\leq 2\binom{i-1}{n^{1/\ell}/\log n}\leq 2\binom{n}{n^{1/\ell}/\log n}\leq 2^{n^{1/\ell}},$$
we then have
\begin{equation}\label{lcycleI2}
M_i\leq \sum_{T}\binom{C(T)}{d_i}\leq\sum_T\binom{\frac{(2^{\ell}+1)n}{d_i^{\ell-1}}}{d_i}
\leq \sum_T 2^{\frac{\ell}{1.88}\left((2^{\ell}+1)en\right)^{1/\ell}}
\leq \sum_T 2^{3\ell n^{1/\ell}}
\leq 2^{\left(3\ell+1\right)n^{1/\ell}}
\end{equation}
for every $i\in \mathcal{I}_2$, where the third inequality is given by Lemma~\ref{optimize}.

Combining (\ref{lcycleI1}) and (\ref{lcycleI2}), we obtain that the number of choices for $S$ is 
\begin{equation*}
\prod_{i=2}^nM_i\leq \prod_{i\in\mathcal{I}_1}M_i\prod_{i\in\mathcal{I}_2}M_i 
\leq 2^{n^{1+1/\ell}}2^{\left(3\ell+1\right)n^{1+1/\ell}}
\leq 2^{\left(3\ell+2\right)n^{1+1/\ell}}.
\end{equation*}
Hence, together with (\ref{girth: minorder}), the total number of certificates is at most
\[
n!(\alpha n^{1/\ell})^n\prod_{i=2}^nM_i
\leq 2^{n^{1+1/\ell}}2^{\left(3\ell+2\right)n^{1+1/\ell}}
\leq 2^{3(\ell+1)n^{1+1/\ell}}
\]
for $n$ sufficiently large, which leads to $|\mathcal{G}_n(2\ell, L)|\leq 2^{3(\ell+1)n^{1+1/\ell}}.$\qed
\section{Hypergraph Enumeration}
In this section, we study the enumeration problems of $r$-graphs with given girth and $r$-graphs without $C^r_4$'s. To prove it, we need a result on the linear Tur\'an number of linear cycles given by Collier-Cartaino, Graber and Jiang~\cite{jiangturan}.
\begin{thm}\cite{jiangturan}\label{turannumber}
For every $r, \ell\geq 3$, there exists a constant $\alpha_{r, \ell}>0$, depending on $r$ and $\ell$, such that 
$$\mathrm{ex}_L(n, C^r_\ell)\leq \alpha_{r,\ell}n^{1+\frac{1}{\lfloor \ell/2\rfloor}}.$$
\end{thm}

\subsection{Proof of Theorem~\ref{evengirth}}
Once we have Theorems~\ref{maintheoremC4} and \ref{mainthmlonger}, it is natural to think about reducing the hypergraph problems to problems on graphs and then apply our graph counting theorems. 
\begin{defi}[Shadow graph]
Given a hypergraph $H$, the shadow graph of $H$, denoted by $\partial_2(H)$, is defined as
$$\partial_2(H)=\{D: |D|=2, \exists e\in H, D\subseteq e\}.$$ 
\end{defi}

\begin{prop}\label{bijection}
Let $r\geq 3$, $\ell\geq 2$ and $H\in \mathrm{Forb}_L(n, r, 2\ell)$. For every $r$-element subset $S\in V(H)$, $S$ forms an $r$-clique in $\partial_2(H)$ if and only if $S$ is a hyperedge in $H$.
\end{prop}
\begin{proof}
Assume that there exists a $r$-clique with vertex set $S$ in $\partial_2(H)$ and two edges $e_1, e_2$ such that $e_1, e_2$ lie on two different hyperedges $f_1$, $f_2$. Without loss of generality, we can assume that $e_1$ and $e_2$ share a common vertex, as otherwise, we let  $e_1=ab$ and $e_2=cd$ and  one of the edge pairs $\{ab, ac\}$ or $\{ac, cd\}$ is contained in different hyperedges.

Let $e_1=ab\subset f_1$ and $e_2=ac\subset f_2$. Note that $c \notin f_1$ and $b\notin f_2$, as otherwise we have $f_1=f_2$ by the linearity of $H$. 
Let $f_3$ be the hyperedge which includes $bc$. Then $f_1, f_2, f_3$ are distinct, and form a $C^r_3$ by the linearity of $H$. This contradicts the fact that $H\in \mathrm{Forb}_L(n, r, 2\ell)$.
\end{proof}

We also need the following short lemma on 4-cycles of the shadow graphs of hypergraphs in $\mathrm{Forb}_L(n, r, 4)$.

\begin{lemma}\label{edgesC4}
For every $r\geq 3$, there exists a constant $\beta=\beta(r)$ such that for every $H\in \mathrm{Forb}_L(n, r, 4)$, the shadow graph $\partial_2(H)$ contains at most $\beta n^{3/2}$ 4-cycles.
\end{lemma}
\begin{proof}
Let $G=\partial_2(H)$. Since the girth of $H$ is larger than 4, every 4-cycle in $G$ must be contained in a hyperedge of $H$. By Theorem~\ref{turannumber}, we have $e(H)\leq \alpha_{r, 4}n^{3/2}$.
Hence, the number of 4-cycles in $G$ is at most $$\binom{r}{4}e(H)\leq \binom{r}{4}\alpha_{r, 4}n^{3/2}=\beta n^{3/2}$$
for $\beta=\binom{r}{4}\alpha_{r, 4}$.
\end{proof}

\noindent\textit{Proof of Theorem~\ref{evengirth} for $\ell=2$.}
Define a map $\varphi: \mathrm{Forb}_L(n, r, 4)\rightarrow \mathcal{G}=\{\partial_2(H): H\in \mathrm{Forb}_L(n, r, 4)\}$ given by $\varphi(H)=\partial_2(H)$. By Proposition~\ref{bijection}, $\varphi$ is a bijection. Note that by Lemma~\ref{edgesC4}, every graph in $\mathcal{G}$ has at most $\beta n^{3/2}$ 4-cycles, where $\beta$ is a constant depending on $r$. Applying Theorem~\ref{maintheoremC4}, when $n$ is sufficiently large, we have  
$$|\mathcal{G}|\leq 2^{11n^{3/2}}.$$
Hence, we obtain that $|\mathrm{Forb}_L(n, r, 4)|=|\mathcal{G}|\leq 2^{11n^{3/2}}$ for $n$ sufficiently large, which completes the proof.\qed
\\

\noindent\textit{Proof of Theorem~\ref{evengirth} for $\ell\geq 3$.}
Define a map $\varphi: \mathrm{Forb}_L(n, r, 2\ell)\rightarrow \mathcal{G}=\{\partial_2(H): H\in \mathrm{Forb}_L(n, r, 2\ell)\}$ given by $\varphi(H)=\partial_2(H)$. By Proposition~\ref{bijection}, $\varphi$ is a bijection. 
For a graph $G=\partial_2(H)\in \mathcal{G}$ and an edge $uv\in E(G)$, since the girth of $H$ is larger than $2\ell$, each $k$-cycle in $G$, which contains edge $uv$,  must be contained in a hyperedge of $H$, for all $3\leq k\leq 2\ell$. Indeed, this hyperedge is unique by the linearity of $H$. Therefore, we have 
$$c_k(u, v; G)\leq \binom{r-2}{k-2}$$
for all $3\leq k\leq 2\ell$. Applying Theorem~\ref{mainthmlonger}, when $n$ is sufficiently large, we have $$|\mathcal{G}|\leq 2^{3(\ell+1)n^{1+1/\ell}}.$$
Hence, we obtain that $|\mathrm{Forb}_L(n, r, 2\ell)|=|\mathcal{G}|\leq 2^{3(\ell+1)n^{1+1/\ell}}$ for $n$ sufficiently large, which completes the proof.\qed

\subsection{Proof of Theorem~\ref{4cycle}}
We now estimate the number of $r$-graphs without $C^r_4$. The main idea is the same as in the previous section: we convert the hypergraph enumeration problem to a graph enumeration problem and then apply Theorem~\ref{maintheoremC4}. However, because of the possible existence of $C^{r}_3$'s, some facts we used before is no longer trivial and even not true. The first difficulty is to give an upper bound on the number of 4-cycles in shadow graphs, and we need the following lemma on the number of $C^r_3$'s.

\begin{lemma}\label{cyclecontaine}
Let $r\geq 3$. For every $H\in \mathrm{Forb}_L(n, C^r_4)$ and every edge $e\in E(H)$, the number of $C^r_3$'s in $H$ containing $e$ as an edge is at most $$\binom{r}{2}(4r^2-10r+7).$$
\end{lemma}
\begin{proof}
For every distinct $u, v\in e$, let $$\mathcal{C}_{u, v}=\{\{e, f_i, g_i\}\subseteq H: e\cap f_i=\{u\}, e\cap g_i=\{v\}, |f_i\cap g_i|=1\}.$$
Suppose $\mathcal{C}_{u,v}$ is nonempty, and fix a $C_0=\{e, f_0, g_0\}\in \mathcal{C}_{u,v}$. For every $C=\{e, f_i, g_i\}\in \mathcal{C}_{u,v}\setminus\{C_0\}$,  we know that $$(f_0\cup g_0)\cap(f_i\cup g_i)-\{u, v\}\neq\emptyset,$$ otherwise, $\{f_0, g_0, f_i, g_i\}$ would form a $C^r_4$. 
Let $w$ be a vertex in $(f_0\cup g_0)\cap(f_i\cup g_i)-\{u, v\}$. Since $w\in f_0\cup g_0-\{u, v\}$, there are at most $2r-3$ choices for $w$. By linearity of $H$, the number of linear 3-cycles in $\mathcal{C}_{u,v}$ containing $w$ is at most $2(r-1)$. Therefore, we get 
\begin{equation}\label{Cuv}
|\mathcal{C}_{u,v}|\leq 1+2(r-1)(2r-3)=4r^2-10r+7.
\end{equation}
Hence, the number of $C^r_3$'s in $H$ containing $e$ as an edge is equal to 
$$\sum_{u, v\in e}|\mathcal{C}_{u, v}|\leq \binom{r}{2}(4r^2-10r+7).$$
\end{proof}
\begin{prop}
For $H\in \mathrm{Forb}_L(n, C^r_4)$, every 4-cycle in $\partial_2(H)$ must be contained in a hyperedge or a $C^r_3$ of $H$.
\end{prop}
\begin{proof}
Assume that a 4-cycle $abcd$ is not contained in any hyperedge of $H$. Then there exist two edges $e_1$ and $e_2$ which lie on two different hyperedges $f_1$ and $f_2$. Without loss of generality, we can assume that $e_1=ab\subset f_1$, and $e_2=ad\subset f_2$. Note that $d \notin f_1$ and $b\notin f_2$, as otherwise we have $f_1=f_2$ by the linearity of $H$. 
Let $f_3$ be the hyperedge which includes $bd$. Then $f_1, f_2, f_3$ are distinct, and form a $C^r_3$ by the linearity of $H$. This contradicts the fact that $H\in \mathrm{Forb}_L(n, C^r_4)$.
\end{proof}

\begin{lemma}\label{c4}
For every $r\geq 3$, there exists a constant $\beta=\beta(r)$ such that for every $H\in \mathrm{Forb}_L(n, C^r_4)$, the shadow graph $\partial_2(H)$ contains at most $\beta n^{3/2}$ 4-cycles.
\end{lemma}
\begin{proof}
Let $G=\partial_2(H)$. 
We first claim that every 4-cycle in $G$ is contained in a hyperedge or a $C^r_3$ of $H$.
By Lemma~\ref{cyclecontaine}, there are at most $$\frac13\binom{r}{2}\left(4r^2-10r+7\right)e(H)$$ $C^r_3$'s in $H$. Since $H$ is linear and contains no $C^r_4$, every 4-cycle in $G$ must be contained in a hyperedge or a $C^r_3$ of $H$. Moreover, by Theorem~\ref{turannumber}, we have $$e(H)\leq\alpha_{r, 4}n^{3/2}.$$
Hence,  the number of 4-cycles in $G$ is at most 
$$3\binom{r}{4}e(H)+3\binom{3r-3}{4}\cdot\frac13\binom{r}{2}(4r^2-10r+7)e(H)\leq \beta n^{3/2}$$
for $$\beta=\left[3\binom{r}{4}+\binom{3r-3}{4}\binom{r}{2}(4r^2-10r+7)\right]\alpha_{r, 4},$$ where $\alpha_{r, 4}$ is a constant defined in Theorem~\ref{turannumber}.
\end{proof}
Another difficulty is that the map we defined in the proof of Theorem~\ref{evengirth} might be no longer injective. To overcome it, we have the following lemma to measure how far the map is from the injection.
\begin{lemma}\label{numberkr}
For every $r\geq 3$, there exists a constant $\alpha=\alpha(r)$ such that for every $H\in \mathrm{Forb}_L(n, C^r_4)$, there are at most $\alpha n^{3/2}$ $r$-cliques in $\partial_2(H)$. 
\end{lemma}
\begin{proof}
Let $G=\partial_2(H)$ and $\mathcal{F}$ be the set of $r$-cliques in $G$. For every $e\in E(H)$, let 
$$\mathcal{F}_{e}=\{F\in \mathcal{F}: |F\cap e|=\max_{f\in E(H)}|F\cap f|\}.$$
Then we have $\mathcal{F}=\bigcup_{e\in H}\mathcal{F}_e.$
Fix an arbitrary hyperedge $e\in H$. For every $2\leq q\leq r,$ let $$\mathcal{R}_q=\{F\in \mathcal{F}_e:  |F\cap e|=q\},$$ then we have $\mathcal{F}_e=\bigcup_{q=2}^r\mathcal{R}_q$.

First, it is trivial to get $|R_r|=1$.
Let $2\leq q\leq r-1$ and $F$ be an $r$-clique in $\mathcal{R}_q$. Since $|F\cap e|=q$, the number of choices for $F\cap e$ is at most $\binom{r}{q}$. Given $F\cap e$, let $u, v$ be two distinct vertices in $F\cap e$. For every $w\in F-e$, by the definition of the shadow graph and the linearity of $H$, there exist hyperedges $f, g$ such that $\{e,f,g\}$ forms a $C^r_3$ with $e\cap f=u, e\cap g=v$ and $f\cap g=w$. 
By (\ref{Cuv}), the number of such $C^r_3$'s is at most $4r^2-10r+7.$
Therefore, the choices of $w$ is at most $4r^2-10r+7$. Hence, we have 
$$\mathcal{R}_q\leq\binom{r}{q}(4r^2-10r+7)^{r-q}.$$
Then, we obtain $$|\mathcal{F}_e|=\sum_{q=2}^r|\mathcal{R}_q|\leq \sum_{r=2}^{r-1}\binom{r}{q}(4r^2-10r+7)^{r-q}+1\leq 2^r(4r^2)^r.$$
Finally, we get 
$$|\mathcal{F}|=\sum_{e\in E(H)}|\mathcal{F}_e|\leq 2^r(4r^2)^re(\mathcal{H})\leq\alpha n^{3/2}$$
for $\alpha=2^r(4r^2)^r\alpha_{r,4},$ where $\alpha_{r,4}$ is the constant defined in Theorem~\ref{turannumber}.
\end{proof}

\noindent\textit{Proof of Theorem~\ref{4cycle}.}
Define a map $\varphi: \mathrm{Forb}_L(n, C^r_4)\rightarrow \mathcal{G}=\{\partial_2(H): H\in \mathrm{Forb}_L(n, C^r_4)\}$ given by $\varphi(H)=\partial_2(H)$. By Lemma~\ref{c4}, every graph $G\in\mathcal{G}$ has at most $\beta n^{3/2}$ 4-cycles, where $\beta$ is a constant depending on $r$. By Theorem~\ref{maintheoremC4}, when $n$ is sufficiently large, we have $$|\mathcal{G}|\leq 2^{11n^{3/2}}.$$
By Lemma~\ref{numberkr}, for every $G\in\mathcal{G}$, the number of $r$-cliques in $G$ is at most $\alpha n^{3/2}$, where $\alpha$ is a constant depending on $r$. Since every hyperedge corresponds to an $r$-clique in its shadow graph, we have 
$$|\varphi^{-1}(G)|\leq 2^{\alpha n^{3/2}}.$$
Finally, we obtain
$$\left|\mathrm{Forb}_L(n, C^r_4)\right|\leq \sum_{G\in\mathcal{G}}\left|\varphi^{-1}(G)\right|\leq|\mathcal{G}|2^{\alpha n^{3/2}}\leq 2^{(11+\alpha)n^{3/2}}$$
for $n$ sufficiently large, which completes the proof.\qed



\end{document}